\documentclass[10pt]{amsart}

\usepackage{amssymb,amsthm,amsmath}
\usepackage[numbers,sort&compress]{natbib}
\usepackage{color}
\usepackage{graphicx}
\usepackage{tikz}
\usepackage{amsfonts}

\title[Full dimensional tori for NLW]{The existence of full dimensional invariant tori for 1-dimensional nonlinear wave equation}

\thanks{The author is supported by the National Natural Science Foundation of China (No.11671066).}
\author{Hongzi Cong}
\address[Hongzi Cong]{School of Mathematical Sciences,
Dalian University of Technology, Dalian, Liaoning 116024, China} \email{conghongzi@dlut.edu.cn}
\author{Xiaoping Yuan}
\address[Xiaoping Yuan]{School of Mathematical Sciences,
Fudan University,
Shanghai 200433,
P. R. China} \email{xpyuan@fudan.edu.cn}


\theoremstyle{plain}
\newtheorem{thm}{Theorem}[section]
 
\newtheorem{lem}[thm]{Lemma}
 
 \theoremstyle{definition}
\newtheorem{defn}[thm]{Definition}
 \theoremstyle{remark}
 \newtheorem{rem}[thm]{Remark}
 
 \numberwithin{equation}{section}
\begin{document}


\begin{abstract}
In this paper we prove the existence and linear stability of full dimensional tori with subexponential decay for 1-dimensional nonlinear wave equation with external parameters, which relies on the method of KAM theory and the idea proposed by Bourgain \cite{BJFA2005}.
\end{abstract}

\maketitle
\section{Introduction and main result}
Consider 1-dimensional nonlinear wave equation (NLW)
\begin{eqnarray}\label{maineq0}
u_{tt}=u_{xx}-V*u-u^3
\end{eqnarray}
on the finite interval $x\in[0,\pi]$ with Dirichlet boundary conditions
\begin{equation*}
u(t,0)=u(t,\pi)=0,\qquad -\infty<t<+\infty,
\end{equation*}
where $V*$ is the Fourier multiplier defined by
\begin{equation*}\widehat{V*u}(n)=V_n\widehat{u}(n)
\end{equation*}
and $\left(V_n\right)_{n\in\mathbb{N}^*}$ are independently chosen in $\left[0,1\right]$, $\mathbb{N}^*=\mathbb{N}\setminus\{0\}$.

To state our results, we need some notations and definitions. Let $z=(z_n)_{n\in\mathbb{N}^*}$ and its complex conjugate $\bar {z}=(\bar z_n)_{n\in\mathbb{N}^*}$. Introduce $I_n=|z_n|^2$ and $J_n=I_n-I_n(0)$, where $I_n(0)$ will be considered as the initial data.
Consider the Hamiltonian $R$ with the following form
 \begin{equation}\label{031901}
 R(z,\bar z)=\sum_{a,k,k'\in{\mathbb{N}^{\mathbb{N}^*}}}B_{akk'}\mathcal{M}_{akk'}
 \end{equation}
with
\begin{eqnarray}\nonumber
\mathcal{M}_{akk'}=\prod_{n\in\mathbb{N}^*}I_n(0)^{a_n}z_n^{k_n}\bar z_n^{k_n'},
\end{eqnarray}
and $B_{akk'}$ are the coefficients.
\begin{defn}\label{020803}
Fixed any $a,k,k'\in\mathbb{N}^{\mathbb{N}^*}$,
denote $(n_i)_{i\geq1}$ the decreasing rearrangement of
\begin{equation*}
\{n:\ \mbox{where $n$ is repeated}\ 2a_n+k_n+k_n'\ \mbox{times}\},
\end{equation*}
i.e.
\begin{equation}\label{020802}
(n_i)_{i\geq1}=(n_1,n_2,\dots,n_l)
\end{equation}
with $l=\sum_{n\in\mathbb{N}^*}(2a_n+k_n+k_n')$ and $n_1\geq n_2\geq n_3\geq \dots$, and $(n_i^*)_{i\geq1}$ the decreasing rearrangement of
\begin{equation*}
\{n:\ \mbox{where $n$ is repeated}\ \left|k_n-k_n'\right|\ \mbox{times}\},
\end{equation*}
i.e.
\begin{equation}
(n_i^*)_{i\geq1}=(n_1^*,n_2^*,\dots,n_j^*)
\end{equation}
with $j=\sum_{n\in\mathbb{N}^*}\left|k_n-k_n'\right|$ and $n_1^*\geq n_2^*\geq n_3^*\geq \dots$.
\end{defn}
\begin{rem}\label{013001.}
Noting that $2a_n+k_n+k_n'\geq \left|k_n-k_n'\right|$, then for any $1\leq i\leq \sum_{n\in\mathbb{N}^{*}}\left|k_n-k_n'\right|$ one has
\begin{equation}\label{021101}
n_i\geq n_i^*,
\end{equation}where $(n_i)_{i\geq 1}$ and $(n_i^*)_{i\geq 1}$ are two decreasing rearrangements defined in Definition \ref{020803}.
\end{rem}

For $n\in\mathbb{N}^*$, take the $ \omega_{n} $ to be random in $[0,1/n]$ and denote $ \|x\| = dist (x,\mathbb{Z})$.
\begin{defn}\label{022005}(Nonresonant Conditions) For any $k,k'\in\mathbb{Z}^{\mathbb{N}^*}$ with $k\neq k'$,
we say $\omega=(\omega_n)_{n\in\mathbb{N}^*}$ is nonresonant in the following sense:
there exists a real number $0<\gamma<1$ such that the following inequalities hold
\begin{eqnarray}
\label{005} \left\| \sum_{n\in\mathbb{N}^*}(k_{n}-k_n')\omega_{n}\right\|\geq \gamma \prod_{n\in \mathbb{N}^*}\frac{1}{1+(k_{n}-k_n')^2n^{5}},
\end{eqnarray}
and
if $n^*_3<n_2^*$ and $\sum_{n\in\mathbb{N}^*}\left|k_n-k_n'\right|\geq3$, then
\begin{eqnarray}
\label{005.} \left\| \sum_{n\in\mathbb{N}^*}(k_{n}-k_n')\omega_{n}\right\|\geq \frac{\gamma^3}{16} \prod_{n\in \mathbb{N}^*\atop n\neq n_1^*,n_2^*}\left(\frac{1}{1+(k_{n}-k_n')^2n^{6}}\right)^4,
\end{eqnarray}
whenever $ 0 \neq k-k' = (k_{n}-k_n')_{n\in \mathbb{N}^*}$ is a finitely supported sequence of integers.
\end{defn}

 Given $\theta\in (0,1)$ and $r>0$, define Banach space $G^{r,\theta}$ of all complex sequences $w=(w_1,w_2,\dots)$ with the finite norm
\begin{eqnarray*}
\label{004} \|w \|_{r,\theta}= \sup_{n\in\mathbb{N}^*}|w_{n}|e^{rn^{\theta}}.
\end{eqnarray*}
Now our main result is as follows:
\begin{thm}\label{Thm}
Given $r>0$, $ 0 < \theta < 1$ and a frequency vector $ \omega = (\omega_{n})_{n\in\mathbb{N}^*}$ satisfying the nonresonant conditions (\ref{005}) and (\ref{005.}),\ then for sufficiently small $ \epsilon > 0 $ there exist $V=(V_n)_{n\in\mathbb{N}^*}$ with $V_n\in\left[0,1\right]$, such that (\ref{maineq0}) has a full dimensional invariant torus $ \mathcal{E}$ with amplitude in $ G^{r,\theta}$ satisfying:\\
(1). the amplitude $I=(I_n)_{n\in\mathbb{N}^*}$ of $ \mathcal{E}$ restricted as
\begin{equation*}
\frac14\epsilon^2e^{-2rn^{\theta}}\leq |I_n|\leq 4\epsilon^2 e^{-2rn^{\theta}};
\end{equation*}
(2). the frequency on  $ \mathcal{E}$ prescribed to be $(n+\omega_n)_{n\in\mathbb{N}^*}$;\\
(3). the invariant tori  $ \mathcal{E}$ linearly stable.
\end{thm}


The existence and linear stability of invariant tori for Hamiltonian PDEs have drawn a lot of concerns during the last decades.  There are many related works for 1-dimensional PDEs. See \cite{BBHM2018,CW1993,K1987,KB2000,KP1996,P1996,W1990,KPB2003,LY2011,BBM2014,BBM2016,FP2015,GL2017,Wang2016} for example. For high dimensional PDEs, Bourgain \cite{BA1998,BB2005} developed a new method initialed by Craig-Wayne \cite{CW1993} to prove the existence of KAM tori for $d$-dimensional nonlinear Schr$\ddot{\mbox{o}}$dinger equations (NLS) and $d$-dimensional NLW with $d\geq 1$, based on the Newton iteration, Fr$\ddot{\mbox{o}}$hlich-Spencer techniques, Harmonic analysis and semi-algebraic set theory. This is so-called C-W-B method.
Later, Eliasson-Kuksin \cite{EK2010} proved a classical KAM theorem which can be applied to $d$-dimensional NLS. It is obtained the existence of KAM tori as well as the linear stability of such tori. Also see \cite{BBNL,BBJEMS,BK2019,Wang2016} for the related problem.

In the above works, the obtained KAM tori are of low dimension which are the support of the quasi-periodic solutions. It must be noted that the constructed quasi-periodic solutions are not typical in the sense that the low dimensional tori have measure zero for any reasonable measure on the infinite dimensional phase space. It is natural at this point to find the full dimensional tori which are the support of the almost periodic solutions. The first result on the existence of almost periodic solutions for 1-dimensional NLW was given by Bourgain in \cite{B1996} using C-W-B method. Later, P$\ddot{\mbox{o}}$schel \cite{P2002} (also see \cite{GX2013} by Geng-Xu) constructed the almost periodic solutions for 1-dimensional NLS by the classical KAM method. These almost periodic solutions were obtained by successive small perturbations of quasi-periodic solutions. To avoid the number of the small divisors increasing fast, the action $I=(I_n)$ must satisfy some very strong compactness properties. In fact, the following super-exponential decay for the action $I$ is given
 $$I_n\sim C^{e^{-|n|}},\qquad C>1,$$
 as $n\rightarrow\infty$. It means that these solutions are with very high regularity and looks like the quasi-periodic ones. Hence, Kuksin raised the following open problem (see Problem 7.1 in \cite{KZ2004}):
 \vskip8pt
 {\it Can the full dimensional KAM tori be expected with a suitable decay, for example, $$I_n\sim|n|^{-C}$$ with some $ C>0$ as $|n|\rightarrow+\infty$ ? }
\vskip8pt

The first try to obtain the existence of full dimensional tori with slower decay was given by Bourgain \cite{BJFA2005}, who proved that 1-dimensional NLS has a full dimensional KAM torus of prescribed frequencies with the actions of the tori obeying the estimates
\begin{equation*}\frac{1}{2}e^{-r{|n|^{1/2}}}\leq I_n\leq 2e^{-r{|n|^{1/2}}}, \qquad r>0.
\end{equation*}
 Recently, Cong-Liu-Shi-Yuan \cite{CLSY} generalized Bourgain's result from $\theta=1/2$ to $0<\theta<1$, i.e the actions of the tori satisfying
\begin{equation*}\label{020203}\frac{1}{2}e^{-r{|n|^{\theta}}}\leq I_n\leq 2e^{-r{|n|^{\theta}}},\qquad \theta\in (0,1)\ \mbox{and}\ r>0.
\end{equation*}
Moreover the authors proved the obtained tori are stable in a sub-exponential long time.

Different from the ideas in \cite{B1996} and \cite{P2002}, Bourgain treated all Fourier modes at once under some suitable Diophantine conditions. See the nonresonant conditions (\ref{005}) for the details, which is similar as the one given in \cite{BJFA2005}. It is well known that the core of KAM theory is how to deal with small divisor. Note that the conditions (\ref{005}) is totally different from the nonresonant conditions used to construct the low dimensional tori, since the factors $n^5$ appears in the denominator, which causes a much worse small denominator problem. Two key observations are given by Bourgain: one is the inequality (\ref{001.}) for $\theta=1/2$; the other is as follows: let
$n_i$ be a finite set of modes satisfying
\begin{equation*}
|n_1|\geq |n_2|\geq |n_3|\geq \cdots
\end{equation*}
and
\begin{equation}\label{022001}
n_1-n_2+n_3-n_4+\cdots=0.
\end{equation}
Note an important fact that in the case of a `near' resonance, there is also a relation
\begin{equation}\label{022002}
n_1^2-n_2^2+n_3^2-n_4^2+\cdots=o(1).
\end{equation}
Unless $n_1=n_2$, from (\ref{022001}) and (\ref{022002}) one has
\begin{equation}\label{022004}|n_1|+|n_2|\leq C\left(|n_3|+|n_4|+\cdots\right),
 \end{equation}where $C$ is a positive constant. In another word, the first two biggest indices $n_1$ and $n_2$ can be controlled by other indices,
which is essential to overcome the small divisor, i.e. giving some good estimate of the solution of homological equation (see Lemma \ref{N3} for the details).

As everyone knows that NLS and NLW are two typical Hamiltonian PDEs which can be considered as touchstones of KAM theory for infinite dimensional Hamiltonian system (see \cite{KP1996} and \cite{P1996}). Some properties of these two equations are similar, but the others are not. A main difference is as follows: for NLS the growth of the frequencies are quadric (also called separation property), while the growth of the frequencies is only linear for NLW. The separation property of the frequencies is essential to control the number of the resonant sets. Eliasson-Kuksin \cite{EK2010} proved a classical KAM theorem which can be applied to $d$-dimensional NLS but not for $d$-dimensional NLW.

In this paper, we would like to study the existence of full dimensional tori for NLW (\ref{maineq0}) with subexponential decay. Our approach and its results are parallel to an investigation of 1-dimensional NLS by Bourgain in \cite{BJFA2005}. Hence some parts of the respective expositions are
quite similar. But we decided to repeat them anyway so that the reader need not refer
to \cite{BJFA2005} for the essentials. One main problem is also there is no separation property for the frequencies of NLW. That is to say the conditions (\ref{022002}) fail, which causes that
the main estimates (\ref{022004}) do not hold all the time. To overcome this difficult we will introduce some new nonresonant conditions firstly.
Precisely we assume that the frequency $\omega$ satisfies a stronger nonresonant conditions (see (\ref{005}) and (\ref{005.}) in Definition \ref{022005} below), which is helpful to control the solution of homological equation (see Lemma \ref{N3} for the details). Of course, we have to show such nonresonant conditions hold for most of $\omega$ in the sense of some measure, which is proven in Lemma \ref{050601}.  Another problem is that we have to show it is possible to choose some parameters $V=(V_n)$ such that the frequency $\omega$ is fixed during the KAM iterations. Different from the case for NLS, the frequency $\omega$ here belongs to $\ell^2$ instead of $\ell^{\infty}$. Therefore, the frequency shift should be calculated carefully to guarantee the inverse function theorem works (see (\ref{114}) for the details). To this end, we introduce the modified norm for the Hamiltonian compared to the one defined in \cite{BJFA2005}, which is based on the regularity of the nonlinear terms for NLW (see Definition \ref{020804} for the details). Also we will give some elementary estimates about this norm. After that, we obtain the existence and linear stability of full dimensional tori with subexponential decay for NLW by a KAM iterative process.

Finally, we also mention a recent work by L. Biasco, J. E. Massetti and M. Procesi \cite{BEP2019}. The authors proved the existence of linear stability of almost periodic solution for 1-dimensional NLS with external parameters with a more geometric point of view by constructing a rather abstract \textbf{counter-term theorem} for infinite dimensional Hamiltonian system. Another interesting byproduct is that a construction of elliptic tori independent of their dimension.

\section{KAM Theorem}
\subsection{Some notations and the norm of the Hamiltonian}
\begin{lem}\label{020801}
Consider the decreasing rearrangement $(n_i)_{i\geq1}$ which is defined by (\ref{020802}) in Definition \ref{020803}
and assume that there are $(\mu_i)_{i\geq 1}$ with $\mu_i\in\{1,-1\}$ such that
\begin{equation}\label{0m}\sum_{n\in\mathbb{N}^*}\mu_in_i=0.\end{equation}
Then for any $0<\theta<1$, one has
\begin{equation}\label{001.}
\sum_{n\in\mathbb{N}^*}(2a_n+k_n+k_n')n^{\theta}-2n_1^{\theta}\geq\left(2-2^{\theta}\right)\sum_{i\geq 3}n_i^{\theta}.
\end{equation}
\end{lem}
\begin{proof}
The proof of (\ref{001.}) is the same as Lemma 2.1 in \cite{CLSY}, which generalizes the result given in Lemma 1.1 in \cite{BJFA2005}.
\end{proof}
\begin{defn}\label{020804}
For any given $ \rho > 0$ and $0<\theta<1$, define the norm of the Hamiltonian
$R$ (see (\ref{031901})) by
\begin{eqnarray}\label{H00}
\| R\|_{\rho}= \sup_{a,k,k'\in\mathbb{N}^{\mathbb{N}^*}}\frac{\left(\prod_{n\in\mathbb{N}^* }n^{\frac12(2a_n+k_n+k_n')}\right)\left|B_{akk'}\right|}{e^{\rho\left(\sum_{n\in\mathbb{N}^*}(2a_{n}+k_{n}+k_{n}')n^{\theta}
-2 n_1^{\theta}\right)}}.
\end{eqnarray}
\end{defn}
For any $k\in \mathbb{N}^{\mathbb{N}^*}$, define
\begin{equation}
\mbox{supp}\ k=\left\{n:k_n\neq 0\right\}.
\end{equation}
Rewrite $R$ as
\begin{equation}\label{N1}
R=R_0+R_1+R_2
\end{equation}
where
\begin{eqnarray}
{R}_0&=&\sum_{a,k,k'\in{\mathbb{N}^{\mathbb{N}^*}}\atop\mbox{supp}\ k\bigcap \mbox{supp}\ k'=\emptyset}B_{akk'}\mathcal{M}_{akk'},\\
{R}_1&=&\sum_{m\in\mathbb{N}^*}J_m\left(\sum_{a,k,k'\in{\mathbb{N}^{\mathbb{N}^*}}\atop\mbox{supp}\ k\bigcap \mbox{supp}\ k'=\emptyset}B_{akk'}^{(m)}\mathcal{M}_{akk'}\right),\\
{R}_2\label{022006}&=&\sum_{m_1,m_2\in\mathbb{N}^*}J_{m_1}J_{m_2}\left(\sum_{a,k,k'\in{\mathbb{N}^{\mathbb{N}^*}}\atop\mbox{no assumption}}B_{akk'}^{(m_1,m_2)}\mathcal{M}_{akk'}\right).
\end{eqnarray}Given $r>0$, let
$$D=\left\{z=(z_n)_{n\in\mathbb{N}^*}:\frac12e^{-rn^{\theta}}\leq |z_n|\leq e^{-rn^{\theta}}\right\},$$
and
$$\Pi=\left\{V=\left(V_n\right)_{n\in\mathbb{N}^*}:V_n\in\left[0,1\right]\right\}.$$
Then we have the following result:
\begin{thm}\label{thm}For $0<\theta<1$ and $r>\frac{100\rho}{2-2^{\theta}}>0$,
suppose the Hamiltonian $$H(z,\bar z)=N(z,\bar z)+\epsilon R(z,\bar z)$$ is real analytic on the domain $D\times \Pi,$
where $$N(z,\bar z)=\sum_{n\in\mathbb{N}^*}\lambda_{n}(V)|z_n|^2$$ is a normal form with
 \begin{equation}\label{020101.}
 \lambda_{n}(V)=\sqrt{n^2+V_{n}},
 \end{equation}
and $R(z,\bar z)$ satisfies
\begin{eqnarray*}
||R||_{\rho}\leq 1.
\end{eqnarray*}
Then given any $\omega=(\omega_n)_{n\in\mathbb{N}^*}$ satisfying the nonresonant conditions (\ref{005}) and (\ref{005.}) and for sufficiently small $\epsilon$ depending on $r,\rho,\theta$ and $\gamma$, there exist $V_*\in \Pi$ and a real analytic symplectic coordinate transformation
$\Phi:D_*\times\left\{V_*\right\}\rightarrow D$, where
\begin{equation*}
{D}_*=\left\{z=(z_n)_{n\in\mathbb{N}^*}:\frac{2}{3}e^{-rn^{\theta}}\leq |z_n|\leq \frac{5}6e^{-rn^{\theta}}\right\}
\end{equation*} satisfying
\begin{eqnarray*}
\sup_{z\in D_*}\left|\left|\left(\Phi-id\right)(z)\right|\right|_{r,\theta}\leq {\epsilon}^{0.4}
\end{eqnarray*}
such that for
${H}_*=H \circ\Phi ={N}_*+{R}_{2,*}$, where
\begin{equation*}
{N}_*=\sum_{n\in\mathbb{N}^*}(n+\omega_n)|z_n|^2
\end{equation*}
and ${R}_{2,*}$ has the form of (\ref{022006}) and satisfies
\begin{equation*}
||{R}_{2,*}||_{10\rho}\leq {\epsilon}^{0.4}.
\end{equation*}
\end{thm}
\subsection{Derivation of homological equations}
The proof of Theorem \ref{thm} employs the rapidly
converging iteration scheme of Newton type to deal with small divisor problems
introduced by Kolmogorov, involving the infinite sequence of coordinate transformations.
At the $s$-th step of the scheme, a Hamiltonian
$H_{s} = N_{s} + R_{s}$
is considered, as a small perturbation of some normal form $N_{s}$ with the form of $$N_{s}=\sum_{n\in\mathbb{N}^*}\lambda_{n,s}(V)|z_n|^2,$$ where
 \begin{equation*}\label{020101}
 \lambda_{n,s}(V)=\sqrt{n^2+\widetilde V_{n,s}(V)}.
 \end{equation*}A transformation $\Phi_{s}$ is
set up so that
$$ H_{s}\circ \Phi_{s} = N_{s+1} + R_{s+1}$$
with another normal form $N_{s+1}$ and a much smaller perturbation $R_{s+1}$. We drop the index $s$ of $H_{s}, N_{s}, R_{s}, \Phi_{s}$ and shorten the index $s+1$ as $+$.

We desire to eliminate the terms $R_0,R_1$ in (\ref{N1}) by the coordinate transformation $\Phi$, which is obtained as the time-1 map $X_F^{t}|_{t=1}$ of a Hamiltonian
vector field $X_F$ with $F=F_0+F_1$. Let ${F}_{0}$ (resp. ${F}_{1}$) has the form of ${R}_0$ (resp. ${R}_{1}$),
that is \begin{eqnarray}
\label{039}&&{F}_0=\sum_{a,k,k'\in{\mathbb{N}^{\mathbb{N}^*}}\atop\mbox{supp}\ k\bigcap \mbox{supp}\ k'=\emptyset}F_{akk'}\mathcal{M}_{akk'},\\
\label{040}&&{F}_1=\sum_{m\in\mathbb{N}^*}J_m\left(\sum_{a,k,k'\in{\mathbb{N}^{\mathbb{N}^*}}\atop\mbox{supp}\ k\bigcap \mbox{supp}\ k'=\emptyset}F_{akk'}^{(m)}\mathcal{M}_{akk'}\right),
\end{eqnarray}
and the homological equations become
\begin{equation}\label{041}
\{N,{F}\}+R_0+R_{1}=[R_0]+[R_1],
\end{equation}
where
\begin{equation}\label{042}
[R_0]=\sum_{a\in\mathbb{N}^{\mathbb{N}^*}}B_{a00}\mathcal{M}_{a00},
\end{equation}
and
\begin{equation}\label{043}
[R_1]=\sum_{m\in\mathbb{N}^*}J_m\sum_{a\in\mathbb{N}^{\mathbb{N}^*}}B_{a00}^{(m)}\mathcal{M}_{a00}.
\end{equation}
The solutions of the homological equations (\ref{041}) are given by
\begin{equation}\label{044}
F_{akk'}=\frac{B_{akk'}}{\sum_{n\in\mathbb{N}^*}(k_n-k^{'}_n){\lambda}_n},
\end{equation}
and
\begin{equation}\label{045}
F_{akk'}^{(m)}=\frac{B_{akk'}^{(m)}}{\sum_{n\in\mathbb{N}^*}(k_n-k^{'}_n)\lambda_n}.
\end{equation}
The new Hamiltonian ${H}_{+}$ has the form
\begin{eqnarray}
H_{+}\nonumber&=&H\circ\Phi\\
&=&\nonumber N+\{N,F\}+R_0+R_1\\
&&\nonumber+\int_{0}^1\{(1-t)\{N,F\}+R_0+R_1,F\}\circ X_F^{t}\ \mathrm{d}{t}
+\nonumber R_2\circ X_F^1\\
&=&\label{046}N_++R_+,
\end{eqnarray}
where
\begin{equation}\label{047}
N_+=N+[R_0]+[R_1],
\end{equation}
and
\begin{equation}\label{048}
R_+=\int_{0}^1\{(1-t)\{N,F\}+R_0+R_1,F\}\circ X_F^{t}\ \mathrm{d} t+R_2\circ X_F^1.
\end{equation}
\subsection{The solvability of the homological equations (\ref{041})}In this subsection, we will estimate
the solutions of the homological equations (\ref{041}). To this end, we define the new norm for the Hamiltonian ${R}$ as follows:
\begin{eqnarray}\label{049}
||{R}||_{\rho}^{+}=\max\left\{||R_0||_{\rho}^{+},||R_1||_{\rho}^{+}|,||R_2||_{\rho}^{+}\right\},
\end{eqnarray}
where
\begin{eqnarray}
\label{050}&&||R_0||_{\rho}^{+}=\sup_{a,k,k'\in\mathbb{N}^{\mathbb{N}^*}}\frac{\left(\prod_{n\in\mathbb{N}^*}n^{\frac{1}2\left(2a_n+k_n+k_n'\right)}\right)\left|B_{akk'}\right|}{e^{\rho(\sum_{n\in\mathbb{N}^*}(2a_n+k_n+k_n')n^{\theta}-2n_1^{\theta})}},\\
\label{051}&&||R_1||_{\rho}^{+}=\sup_{a,k,k'\in\mathbb{N}^{\mathbb{N}^*}\atop m\in\mathbb{N}^*}\frac{\left(\prod_{n\in\mathbb{N}^*}n^{\frac12\left(2a_n+k_n+k_n'\right)}\right)\left|mB^{(m)}_{akk'}\right|}{e^{\rho(\sum_{n\in\mathbb{N}^*}(2a_n+k_n+k_n')n^{\theta}+2m^{\theta}-2n_1^{\theta})}},\\
\label{052}&&||R_2||_{\rho}^{+}=\sup_{a,k,k'\in\mathbb{N}^{\mathbb{N}^*}\atop
m_1,m_2\in\mathbb{N}^*}\frac{\left(\prod_{n\in\mathbb{N}^*}n^{\frac12\left(2a_n+k_n+k_n'\right)}\right)\left|m_1m_2B^{(m_1,m_2)}_{akk'}\right|}{e^{\rho(\sum_{n\in\mathbb{N}^*}(2a_n+k_n+k_n')n^{\theta}
+2m_1^{\theta}+2m_2^{\theta}-2n_1^{\theta})}}.
\end{eqnarray}
Moreover, one has the following estimates:
\begin{lem}\label{N2}
Given any $ \delta,\rho>0,$ one has
\begin{equation}\label{053}
||R||_{\rho+\delta}^{+}\leq\left(\frac{1}{\delta}\right)^{ C(\theta)\delta^{-\frac{1}{\theta}}}||R||_{\rho}
\end{equation}
and
\begin{equation}\label{054}
||R||_{\rho+\delta}\leq\frac{C(\theta)}{\delta^2}||R||_{\rho}^{+},
\end{equation}
where $C(\theta)$ is a  positive constant depending on $\theta$ only.
\end{lem}
\begin{proof}
Firstly, we will prove the inequality (\ref{053}).
Write $\mathcal{M}_{akk'}$ in the form of
\begin{equation*}
\mathcal{M}_{akk'}=\mathcal{M}_{abll'}=\prod_{n\in\mathbb{N}^*}I_n(0)^{a_n}I_n^{b_n}z_n^{l_n}{\bar z_n}^{l_n'},
\end{equation*}
where
\begin{equation*}
b_n=k_n\wedge k_n',\quad l_n=k_n-b_n,\quad l_n'=k_n'-b_n
\end{equation*}
and
$l_nl_n'=0$ for all $n$.

Express the term
\begin{equation*}
\prod_{n\in\mathbb{N}^*}I_n^{b_n}=\prod_{n\in\mathbb{N}^*}(I_n(0)+J_n)^{b_n}
\end{equation*}by the monomials of the form
\begin{equation*}
\prod_{n\in\mathbb{N}^*}I_n(0)^{b_n},
\end{equation*}
\begin{equation*}
\sum_{m,b_m\geq 1}\left(I_m(0)^{b_m-1}J_m\right)\left(\prod_{n\neq m}I_n(0)^{b_n}\right),
\end{equation*}
\begin{equation*}
\sum_{m,b_m\geq2\atop
r\leq b_m-2}\left(\prod_{n< m}I_n(0)^{b_n}\right)\left(I_m(0)^{r}J_m^2I_m^{b_m-r-2}\right)\left(\prod_{n> m}I_n^{b_n}\right),
\end{equation*}
and
\begin{eqnarray*}
&&\sum_{m_1< m_2,b_{m_1},b_{m_2}\geq 1\atop
r\leq b_{m_2}-1}\left(\prod_{n< m_1}I_n(0)^{b_n}\right)\left(I_{m_1}(0)^{b_{m_1}-1}J_{m_1}\right)
\\
&&\nonumber\times\left(\prod_{m_1<n< m_2}I_n(0)^{b_n}\right)\left(I_{m_2}(0)^{r}J_{m_2}I_{m_2}^{b_{m_2}-r-1}\right)
\left(\prod_{n> m_2}I_n^{b_n}\right).
\end{eqnarray*}

Now we will estimate the bounds for the coefficients respectively.

 Consider the term
$\mathcal{M}_{akk'}=\prod_{n\in\mathbb{N}^*}I_n(0)^{a_n}z_n^{k_n}\bar z_n^{k_n'}$ with fixed $a,k,k'$ satisfying $k_nk_n'=0$ for all $n$.
It is easy to see that $\mathcal{M}_{akk'}$ comes from some parts of the terms $\mathcal{M}_{\alpha\kappa\kappa'}$ with no assumption for $\kappa$ and $\kappa'$. For any given $n$ one has
\begin{equation*}
I_n(0)^{a_n}z_n^{k_n}\bar z_n^{k_n'}=\sum_{\beta_n=\kappa_n\wedge \kappa_n'}I_n(0)^{\alpha_n+\beta_n}z_n^{\kappa_n-\beta_n}\bar z_n^{\kappa_n'-\beta_n},
\end{equation*}
where
\begin{equation}\label{055}
\alpha_n+\beta_n=a_n,
\end{equation}
and
\begin{equation}\label{056}
\kappa_n-\beta_n=k_n,\qquad \kappa_n'-\beta_n=k_n'.
\end{equation}
Hence one has
\begin{equation}\label{012902}
2a_n+k_n+k_n'=2(\alpha_n+\beta_n)+(\kappa_n-\beta_n)+(\kappa_n'-\beta_n)=2\alpha_n+\kappa_n+\kappa_n'.
\end{equation}
Moreover, if $0\leq\alpha_n\leq a_n$ is chosen, so $\beta_n,k_n,k_n'$ are determined.
On the other hand,
\begin{eqnarray}
&&\nonumber\left|\left(\prod_{n\in\mathbb{N}^*}n^{\frac12\left(2\alpha_n+\kappa_n+\kappa_n'\right)}\right)B_{\alpha\kappa\kappa'}\right|\\
&\leq&\nonumber ||R||_{\rho}e^{\rho\left(\sum_{n\in\mathbb{N}^*}(2\alpha_n+\kappa_n+\kappa_n')n^{\theta}-2\nu_1^{\theta}\right)}\\
&=&\nonumber||R||_{\rho}e^{\rho\left(\sum_{n\in\mathbb{N}^*}(2a_n+k_n+k_n')n^{\theta}-2n_1^{\theta}\right)},
\end{eqnarray}
where the last equality is based on (\ref{012902}) and
\begin{equation*}
\nu_1=\max \mbox{supp}\ \alpha+\kappa+\kappa'.
\end{equation*}
Hence,
\begin{equation}\label{057}
\left|\left(\prod_{n\in\mathbb{N}^*}n^{\frac12\left(2a_n+k_n+k_n'\right)}\right)B_{akk'}\right|\leq||R||_{\rho}e^{\rho\left(\sum_{n\in\mathbb{N}^*}(2a_n+k_n+k_n')n^{\theta}-2n_1^{\theta}\right)}\left(\prod_{n\in\mathbb{N}^*}(1+a_n)\right).
\end{equation}
In view of (\ref{050}) and (\ref{057}), we have
\begin{eqnarray}
||R_0||_{\rho+\delta}^{+}\nonumber
&\leq&||R||_{\rho}e^{-\delta\left(\sum_{n\in\mathbb{N}^*}(2a_n+k_n+k_n')n^{\theta}
-2n_1^{\theta}\right)}\left(\prod_{n\in\mathbb{N}^*}(1+a_n)\right)\\
&\leq&\label{058}||R||_{\rho}\left(\frac{1}{\delta}\right)^{ C(\theta)\delta^{-\frac{1}{\theta}}},
\end{eqnarray}
where the last inequality is based on (7.37) in \cite{CLSY} and $C(\theta)$ is a positive constant depending only on $\theta$.

Next consider the term
$J_m\mathcal{M}_{akk'}=J_m\prod_{n\in\mathbb{N}^*}I_n(0)^{a_n}z_n^{k_n}\bar z_n^{k_n'}$ with fixed $a,k,k'$ satisfying $k_nk_n'=0$ for all $n$.
The term $J_m\mathcal{M}_{akk'}$ also comes from some parts of the terms $\mathcal{M}_{\alpha\kappa\kappa'}$ with no assumption for $\kappa$ and $\kappa'$.

For any given $n\neq m$ one has
\begin{equation*}
I_n(0)^{a_n}z_n^{k_n}\bar z_n^{k_n'}=\sum_{\beta_n=\kappa_n\wedge \kappa_n'}I_n(0)^{\alpha_n+\beta_n}z_n^{\kappa_n-\beta_n}\bar z_n^{\kappa_n'-\beta_n}.
\end{equation*}
Following (\ref{055}), (\ref{056}) and (\ref{012902}), one has
\begin{equation*}
\alpha_n+\beta_n=a_n,
\end{equation*}
\begin{equation*}
\kappa_n-\beta_n=k_n,\qquad \kappa_n'-\beta_n=k_n',
\end{equation*}
and
\begin{equation*}
2\alpha_n+\kappa_n+\kappa_n'=2a_n+k_n+k_n'.
\end{equation*}
Moreover, if $0\leq\alpha_n\leq a_n$ is chosen, so $\beta_n,k_n,k_n'$ are determined.

For any given $n=m$ one has
\begin{equation*}
J_mI_m(0)^{a_m}q_m^{k_m}\bar q_m^{k_m'}=\sum_{\beta_m=\kappa_m\wedge \kappa_m'}\beta_mJ_mI_m(0)^{\alpha_m+\beta_m-1}q_m^{\kappa_m-\beta_m}\bar q_m^{\kappa_m'-\beta_m}.
\end{equation*}
Hence,
\begin{equation}\label{022101}
\alpha_m+\beta_m-1=a_m,
\end{equation}
\begin{equation*}
\kappa_m-\beta_m=k_m,\qquad \kappa_m'-\beta_m=k_m',
\end{equation*}
and
\begin{equation}\label{012903}
2\alpha_m+\kappa_m+\kappa_m'=2a_m+k_m+k_m'+2.
\end{equation}
Moreover, if $0\leq\alpha_m\leq a_m$ is chosen (noting that $\alpha_m=a_m+1\Leftrightarrow \beta_m=0$), so $\beta_m,k_m,k_m'$ are determined.
On the other hand,
\begin{eqnarray}
&&\nonumber\left|\left(\prod_{n\in\mathbb{N}^*}n^{\frac12\left(2\alpha_n+\kappa_n+\kappa_n'\right)}\right)B_{\alpha\kappa\kappa'}\right|\\
&\leq&\nonumber ||R||_{\rho}e^{\rho\left(\sum_{n\in\mathbb{N}^*}(2\alpha_n+\kappa_n+\kappa_n')n^{\theta}-2\nu_1^{\theta}\right)}\\
&\leq&\nonumber||R||_{\rho}e^{\rho\left(\sum_{n\in\mathbb{N}^*}(2a_n+k_n+k_n')n^{\theta}+2m^{\theta}-2n_1^{\theta}\right)},
\end{eqnarray}
where the last equality is based on (\ref{012903}) and $n_1\leq \nu_1$.
Then
\begin{eqnarray}
&&\nonumber\left|\left(\prod_{n\in\mathbb{N}^*}n^{\frac12\left(2a_n+k_n+k_n'\right)}\right)mB_{akk'}^{(m)}\right|\\
&\leq& \nonumber ||R||_{\rho}e^{\rho\left(\sum_{n\in\mathbb{N}^*}(2a_n+k_n+k_n')n^{\theta}+2m^{\theta}-2n_1^{\theta}\right)}\left(\prod_{n\in\mathbb{N}^*}(1+a_n)\right)\beta_m\\
&\leq& \nonumber ||R||_{\rho}e^{\rho\left(\sum_{n\in\mathbb{N}^*}(2a_n+k_n+k_n')n^{\theta}+2m^{\theta}-2n_1^{\theta}\right)}\left(\prod_{n\in\mathbb{N}^*}(1+a_n)\right)\left(1+a_m\right)\\
&&\nonumber \mbox{(based on (\ref{022101}))}\\
&\leq&||R||_{\rho}\left(\frac{1}{\delta}\right)^{ C(\theta)\delta^{-\frac{1}{\theta}}},
\end{eqnarray}
where the last inequality is based on (7.37) in \cite{CLSY} and $C(\theta)$ is a positive constant depending only on $\theta$. In view of (\ref{051}), one has
\begin{equation}\label{058.}
||R_1||_{\rho+\delta}^{+}
\leq||R||_{\rho}\left(\frac{1}{\delta}\right)^{ C(\theta)\delta^{-\frac{1}{\theta}}}.
\end{equation}
Similarly, one has
\begin{eqnarray}\label{058..}
||R_2||_{\rho+\delta}^{+}\leq\left(\frac{1}{\delta}\right)^{C(\theta)\delta^{-\frac{1}{\theta}}}||R||_{\rho}.
\end{eqnarray}
In view of (\ref{049}), (\ref{058}), (\ref{058.}) and (\ref{058..}), we finish the proof of (\ref{053}).

On the other hand, the coefficient of $\mathcal{M}_{abll'}$ increases by at most a factor $$\left(\sum_{n\in\mathbb{N}^*}(a_n+b_n)\right)^2,$$ then one has
\begin{eqnarray}
\nonumber||R||_{\rho+\delta}
&\leq&\nonumber||R||_{\rho}^{+}\left(\sum_{n\in\mathbb{N}^*}(a_n+b_n)\right)^2
e^{-\delta\left(\sum_{n\in\mathbb{N}^*}(2a_n+k_n+k_n')n^{\theta}-2n_1^{\theta}\right)}\\
\nonumber&\leq&||R||_{\rho}^{+}\left(2\sum_{i\geq 3}n_i^{\theta}\right)^2 e^{-\delta(2-2^\theta)\sum_{i\geq3}n_i^{\theta}}\\
\nonumber &\leq&\frac{4}{(2-2^{\theta})^2\delta^2}||R||_{\rho}^{+},
\end{eqnarray}
where the last inequality is based on Lemma 7.5 in \cite{CLSY} with $p=2$, and we finish the proof of (\ref{054}).
\end{proof}

\begin{lem}\label{N3}
Assume $\omega=(\omega_n)_{n\in\mathbb{N}^*}$ with $\omega_n=\lambda_n-n$ satisfies the nonresonant conditions (\ref{005}) and (\ref{005.}). Then for any $\rho>0,0<\delta\ll1$ (depending only on $\theta$), the solutions of the homological equations (\ref{041}), which are given by (\ref{044}) and (\ref{045}), satisfy
\begin{eqnarray}\label{061}
\left|\left|{F}\right|\right|_{\rho+\delta}^{+}\leq \frac{1}{\gamma^3}\cdot e^{C(\theta)\delta^{-\frac5\theta}}||{{R}}||_{\rho}^{+},
\end{eqnarray}
 where $C(\theta)$ is a positive constant depending on $\theta$ only.
\end{lem}
\begin{proof}

We distinguish two cases:\\

\textbf{Case. 1.} $n_3^*<n_2^*$. \\
Since $$\sum_{n\in\mathbb{N}^*}(k_n-k'_n)n\in\mathbb{Z},$$
the nonresonant conditions (\ref{005.}) implies
\begin{equation}\label{063.}
\left|\sum_{n\in\mathbb{N}^*}(k_n-k'_n)\lambda_n\right|\geq
\frac{\gamma^3}{16}\prod_{n\in\mathbb{N}^*\atop n\neq n^*_1,n_2^*}\left(\frac{1}{1+{(k_n-k'_n)^2}n^6}\right)^4.
\end{equation}
Hence,
\begin{eqnarray*}
\nonumber&&{\left(\prod_{n\in\mathbb{N}^*}n^{\frac12\left(2a_n+k_n+k_n'\right)}\right)\left|{F}_{akk'}\right|}e^{-(\rho+\delta)\left(\sum_{n\in\mathbb{N}^*}2a_n+k_n+k'_n)n^{\theta}-2n_1^{\theta}\right)}\\
&=&\nonumber\frac{\left(\prod_{n\in\mathbb{N}^*}n^{\frac12\left(2a_n+k_n+k_n'\right)}\right)\left|{B}_{akk'}\right|}{|\sum_{n\in\mathbb{N}^*}(k_n-k'_n)\lambda_n|}\times e^{-(\rho+\delta)(\sum_{n\in\mathbb{N}^*}(2a_n+k_n+k'_n)n^{\theta}-2n_1^{\theta})}\\
&&\nonumber\mbox{(in view of (\ref{044}))}\\
\nonumber&\leq& \frac{16}{\gamma^{3}}||{R_0}||_{\rho}^{+} \left(\prod_{n\in\mathbb{N}^*\atop n\neq n_1^*,n_2^*}\left({1+(k_n-k'_n)^2n^6}\right)^4\right)\times e^{-\delta\left(\sum_{n\in\mathbb{N}^*}(2a_n+k_n+k'_n)n^{\theta}-2n_1^{\theta}\right)}\\
\nonumber&&
 \mbox{(in view of (\ref{050}) and (\ref{063.}))}\\
&\leq&\label{064.}\frac{1}{\gamma^3}\cdot e^{C_1(\theta)\delta^{-\frac5\theta}} ||{R_0}||_{\rho}^+,
\end{eqnarray*}
where the last inequality is based on Lemma \ref{031902} and $ C_1(\theta)$ is a positive constant depending on $\theta$ only, which finishes the proof of \begin{eqnarray}\label{031903}
\left|\left|{F}_0\right|\right|_{\rho+\delta}^{+}\leq \frac{1}{\gamma^3}\cdot e^{C(\theta)\delta^{-\frac5\theta}}||{{R_0}}||_{\rho}^{+}
\end{eqnarray}
in \textbf{Case. 1}.

$\textbf{Case. 2.}$ $n_3^*=n_2^*$.

If $$n_1^*\geq 10\sum_{i\geq2}n_i^*,$$ then one has
\begin{equation*}
\left|\sum_{n\in\mathbb{N}^*}(k_n-k'_n)\lambda_n\right|\geq
1,
\end{equation*}where there is no small divisor. Hence we always assume that
\begin{equation*}
n_1^*< 10\sum_{i\geq2}n_i^*.
\end{equation*}
In view of $n_3^*=n_2^*$, then one has
\begin{equation*}
n_1^*\leq 11\sum_{i\geq 3}n_i^*
\end{equation*}
and
\begin{equation*}
\left(n_1^*\right)^{{\theta}}\leq 11^{{\theta}}\sum_{i\geq 3}\left(n_i^*\right)^{{\theta}}.
\end{equation*}
Moreover,
\begin{eqnarray}
\nonumber\sum_{i\geq 1}\left(n_i^*\right)^{\theta}&\leq& \left(11^{\theta}+2\right)\sum_{i\geq 3}\left(n_i^*\right)^{\theta}\\
&\leq& \label{013002}{\left(11^{\theta}+2\right)}\sum_{i\geq 3}n_i^{\theta}\qquad (\mbox{in view of Remark \ref{013001.}})\\
&\leq&\frac{11^{\theta}+2}{2-2^{\theta}}\left(\sum_{n\in\mathbb{N}^*}(2a_n+k_n+k'_n)n^{\theta}-2n_1^{\theta}\right),
\end{eqnarray}
where the last inequality is based on (\ref{001.}).

Since $$\sum_{n\in\mathbb{N}^*}(k_n-k'_n)n\in\mathbb{Z},$$
the nonresonant conditions (\ref{005}) implies
\begin{equation}\label{063}
\left|\sum_{n\in\mathbb{N}^*}(k_n-k'_n)\lambda_n\right|\geq
\gamma\prod_{n\in\mathbb{N}^*}\frac{1}{1+{(k_n-k'_n)^2}n^5}.
\end{equation}
Following the proof of (\ref{064.}) one has
\begin{eqnarray*}
\nonumber&&{\left(\prod_{n\in\mathbb{N}^*}n^{\frac12\left(2a_n+k_n+k_n'\right)}\right)\left|{F}_{akk'}\right|}e^{-(\rho+\delta)\left(\sum_{n\in\mathbb{N}^*}\left(2a_n+k_n+k'_n\right)n^{\theta}-2n_1^{\theta}\right)}\\
&\leq&\label{064}\frac{1}{\gamma}\cdot e^{C_2(\theta)\delta^{-\frac5\theta}} ||{R_0}||_{\rho}^+,
\end{eqnarray*}
where $C_2(\theta)$ is a positive constant depending on $\theta$ only, which finishes the proof of \begin{eqnarray}\label{031904}
\left|\left|{F}_0\right|\right|_{\rho+\delta}^{+}\leq \frac{1}{\gamma^3}\cdot e^{C(\theta)\delta^{-\frac5\theta}}||{{R_0}}||_{\rho}^{+}
\end{eqnarray}
in \textbf{Case. 2}.

Similarly, one can prove \begin{eqnarray}\label{031905}
\left|\left|{F}_1\right|\right|_{\rho+\delta}^{+}\leq \frac{1}{\gamma^3}\cdot e^{C(\theta)\delta^{-\frac5\theta}}||{{R_1}}||_{\rho}^{+}.
\end{eqnarray}

In view of (\ref{031903}), (\ref{031904}) and (\ref{031905}), we finish the proof of (\ref{061}).
\end{proof}

\subsection{The new perturbation $R_+$ and the new normal form $N_+$}

Recall the new term $R_+$ is given by (\ref{048}) and
write
\begin{equation}\label{069}
R_+=R_{0+}+R_{1+}+R_{2+}.
\end{equation}
Now we will estimate $R_{i+}$ for $i=0,1,2$ respectively. To this end, we give the following estimate:
\begin{lem}\label{012901}(\textbf{Poisson Bracket})
Let $\theta\in(0,1),\rho>0$ and $0<\delta_1,\delta_2\ll1$ (depending on $\theta,\rho$). Then one has
\begin{equation}\label{042704}
\left|\left|\left\{H_1,H_2\right\}\right|\right|_\rho\leq \frac{1}{\delta_2}\left(\frac{1}{\delta_1}\right)^{C({\theta}){\delta_1^{-\frac{1}{\theta}}}}\left|\left|H_1\right|\right|_{\rho-\delta_1}||H_2||_{\rho-\delta_2},
\end{equation}
where $C(\theta)$ is a positive constant depending on $\theta$ only.
\end{lem}
\begin{proof}
Let
\begin{equation*}
H_1=\sum_{a,k,k'\in\mathbb{N}^{\mathbb{N}^*}} b_{akk'}\mathcal{M}_{akk'}
\end{equation*}
and
\begin{equation*}
H_2=\sum_{A,K,K'\in\mathbb{N}^{\mathbb{N}^*}} B_{AKK'}\mathcal{M}_{AKK'}.
\end{equation*}
It follows easily that
\begin{equation*}
\{H_1,H_2\}=\sum_{a,k,k',A,K,K'\in\mathbb{N}^{\mathbb{N}^*}}b_{akk'}B_{AKK'}\{\mathcal{M}_{akk'},\mathcal{M}_{AKK'}\},
\end{equation*}
where
\begin{eqnarray*}
\{\mathcal{M}_{akk'},\mathcal{M}_{AKK'}\}
&=&-\textbf{i}\sum_{j\in\mathbb{N}^*}\left(\prod_{n\neq j}I_n(0)^{a_n+A_n}z_n^{k_n+K_n}\bar{z}_n^{k_n'+K_n'}\right)\\
&&\times\left((k_jK_j'-k_j'K_j)I_j(0)^{a_j+A_j}z_j^{k_j+K_j-1}\bar{z}_j^{k_j'+K_j'-1}\right).
\end{eqnarray*}
Then the coefficient of
\[
\mathcal{M}_{\alpha\kappa\kappa'} :=\prod_{n\in\mathbb{N}^*}I_n(0)^{\alpha_n}z_n^{\kappa_n}\bar{z}_n^{\kappa'_n}
\]
is given by
\begin{equation}\label{006*}
B_{\alpha\kappa\kappa'}= -{\textbf{i}}\sum_{j\in\mathbb{N}^*}\sum_{*}\sum_{**}(k_jK_j'-k_j'K_j)b_{akk'}B_{AKK'},
\end{equation}
where
\begin{equation*}
\sum_{*}=\sum_{a,A \atop a+A=\alpha},
\end{equation*}
and
\begin{equation}\label{020805}
\sum_{**}=\sum_{k,k',K,K'\atop \mbox{when}\ n\neq j, k_n+K_n=\kappa_n,k_n'+K_n'=\kappa_n';\mbox{when}\ n=j, k_n+K_n-1=\kappa_n,k_n'+K_n'-1=\kappa_n'}.
\end{equation}
In view of (\ref{H00}) in Definition \ref{020804}, one has
\begin{eqnarray}
\nonumber &&\left|\left(\prod_{n\in\mathbb{N}^*}n^{\frac12\left(2a_n+k_n+k_n'\right)}\right)b_{akk'}\right|\\
 \nonumber &\leq& ||H_1||_{\rho-\delta_1}e^{(\rho-\delta_1)\left(\sum_{n\in\mathbb{N}^*}(2a_n+k_n+k_n')n^{\theta}-2n_1^{\theta}\right)}\\
&=&||H_1||_{\rho-\delta_1}e^{\rho\left(\sum_{n\in\mathbb{N}^*}(2a_n+k_n+k_n')n^{\theta}-2 n_1^{\theta}\right)}
e^{-\delta_1\left(\sum_{n\in\mathbb{N}^*}(2a_n+k_n+k_n')n^{\theta}-2n_1^{\theta}\right)}\nonumber\\
\label{007}&\leq&||H_1||_{\rho-\delta_1}e^{\rho\left(\sum_{n\in\mathbb{N}^*}
(2a_n+k_n+k_n')n^{\theta}-2n_1^{\theta}\right)}e^{-(2-2^{\theta})\delta_1
\sum_{i\geq 3}n_i^{\theta}},
\end{eqnarray}
where the last inequality is based on (\ref{001.}) in Lemma \ref{0m}.

Similarly,
\begin{equation}
\left|\prod_{n\in\mathbb{N}^*}n^{\frac12\left(2A_n+K_n+K_n'\right)}B_{akk'}\right|\leq \label{008}||H_2||_{\rho-\delta_2}e^{\rho\left(\sum_{n\in\mathbb{N}^*}(2A_n+K_n+K_n')
n^{\theta}-2N_1^{\theta}\right)}e^{-(2-2^{\theta})\delta_2\sum_{i\geq 3}N_i^{\theta}}.
\end{equation}
Substituting (\ref{007}) and (\ref{008}) in (\ref{006*}) gives
\begin{eqnarray}
\label{009}&&\nonumber\left|\prod_{n\in{\mathbb{N}^*}}n^{\frac12\left(2\alpha_n+\kappa_n+\kappa_n'\right)}B_{\alpha\kappa\kappa'}\right|\\
&\leq&\nonumber\left|\prod_{n\in{\mathbb{N}^*}}n^{\frac12\left(2a_n+k_n+k_n'+2A_n+K_n+K_n'\right)}B_{\alpha\kappa\kappa'}\right|\\ &&\nonumber\mbox{(in view of $\alpha_n=a_n+A_n$ and (\ref{020805}))}\\
\nonumber&\leq&\nonumber||H_1||_{\rho-\delta_1}||H_2||_{\rho-\delta_2}\sum_{j\in\mathbb{N}^*}\sum_{*}\sum_{**}|k_jK_j'-k_j'K_j|\\
&&\times\nonumber e^{\rho\left(\sum_{n\in\mathbb{N}^*}(2a_n+k_n+k_n')n^{\theta}-2n_1^{\theta}
+\sum_{n\in\mathbb{N}^*}(2A_n+K_n+K_n')n^{\theta}-2N_1^{\theta}\right)}\\
\nonumber&&\nonumber\times e^{-(2-2^{\theta})\delta_1\sum_{i\geq 3}n_i^{\theta}}e^{-(2-2^{\theta})\delta_2\sum_{i\geq 3}N_i^{\theta}}\\
\nonumber&=&\nonumber||H_1||_{\rho-\delta_1}||H_2||_{\rho-\delta_2}\sum_{j\in\mathbb{N}^*}\sum_{*}\sum_{**}|k_jK_j'-k_j'K_j|\\
&&\times\nonumber e^{\rho\left(\sum_{n\in\mathbb{N}^*}(2\alpha_n+\kappa_n+\kappa_n')
n^{\theta}+2j^{\theta}\right)}e^{-2\rho n_1^{\theta}-2\rho N_1^{\theta}}\\
\nonumber&&\times \nonumber e^{-(2-2^{\theta})\delta_1\sum_{i\geq 3}n_i^{\theta}}e^{-(2-2^{\theta})\delta_2\sum_{i\geq 3}N_i^{\theta}}\\
&=&\nonumber||H_1||_{\rho-\delta_1}||H_2||_{\rho-\delta_2}
e^{\rho\left(\sum_{n\in\mathbb{N}^*}(2\alpha_n+\kappa_n+\kappa_n')n^{\theta}-2\nu_1^{\theta}\right)}\\
&&\nonumber\times\sum_{j\in\mathbb{N}^*}\sum_{*}\sum_{**}|k_jK_j'-k_j'K_j| e^{2\rho\left(j^{\theta}+\nu_1^{\theta}-n_1^{\theta}-N_1^{\theta}\right)}\\
&&\label{020806}\times e^{-(2-2^{\theta})\delta_1\sum_{i\geq 3}n_i^{\theta}}e^{-(2-2^{\theta})\delta_2\sum_{i\geq 3}N_i^{\theta}},
\end{eqnarray}
where
\begin{equation*}
\nu_1=\max\{n:\alpha_n+\kappa_n+\kappa_n'\neq0\}.
\end{equation*}
Following the proof of (4.8) in Lemma 4.1 in \cite{CLSY}, one has
\begin{equation}\label{020807}
I\leq \frac{1}{\delta_2}\left(\frac{1}{\delta_1}\right)^{C({\theta})
{\delta_1^{-\frac{1}{\theta}}}},
\end{equation}
where
\begin{eqnarray*}
I&=&\sum_{j}\sum_{*}\sum_{**}|k_jK_j'-k_j'K_j|
e^{2\rho\left(j^{\theta}+\nu_1^{\theta}-n_1^{\theta}-N_1^{\theta}\right)}\\
\nonumber&&\times e^{-(2-2^{\theta})\delta_1\sum_{i\geq 3}n_i^{\theta}}
e^{-(2-2^{\theta})\delta_2\sum_{i\geq 3}N_i^{\theta}}.
\end{eqnarray*}
Hence in view of (\ref{020806}) and (\ref{020807}), we finish the proof of (\ref{042704}).
\end{proof}
Based on Lemma \ref{012901} and following the proof of (4.54)-(4.56) in \cite{CLSY}, one has
\begin{eqnarray}
||R_{0+}||_{\rho+3\delta}^{+}
\label{0861} &\leq& \frac1{\gamma^3}\cdot e^{{\delta^{-\frac{10}{\theta}}}}\left(||R_0||_{\rho}^++||R_1||_{\rho}^+\right)\left(||R_0||_{\rho}^+
+{||R_1||_{\rho}^+}^2\right),\\
||R_{1+}||_{\rho+3\delta}^{+}
\label{0862}&\leq& \frac1{\gamma^3}\cdot e^{{\delta^{-\frac{10}{\theta}}}}\left(||R_0||_{\rho}^++{||R_1||_{\rho}^+}^2\right),\\
||R_{2+}||_{\rho+3\delta}^{+}
\label{0863}&\leq& ||R_2||_{\rho}^++\frac1{\gamma^3}\cdot e^{{\delta^{-\frac{10}{\theta}}}}\left(||R_0||_{\rho}^++||R_1||_{\rho}^+\right).
\end{eqnarray}

The new normal form $N_+$ is given in (\ref{047}). Note that $[R_0]$ (in view of (\ref{042})) is a constant which does not affect the Hamiltonian vector field. Moreover, in view of (\ref{043}), we denote by
\begin{equation*}\label{087}
\omega_{n+}=\sqrt{n^2+ \widetilde{V}_n}+\sum_{a\in\mathbb{N}^{\mathbb{N}^*}}B_{a00}^{(n)}\mathcal{M}_{a00},
\end{equation*}
where the terms $$\sum_{a\in\mathbb{N}^{\mathbb{N}^*}}B_{a00}^{(n)}\mathcal{M}_{a00}$$ is the so-called frequency shift. The estimate of $\left|\sum_{a\in\mathbb{N}^{\mathbb{N}^*}}B_{a00}^{(n)}\mathcal{M}_{a00}\right|$ will be given in the next section (see (\ref{114}) for the details).

Finally, we give the estimate of the Hamiltonian vector field.
\begin{lem}\label{H6}
Given a Hamiltonian
\begin{equation*}\label{028}
H=\sum_{a,k,k'\in\mathbb{N}^{\mathbb{N}^*}}B_{akk'}\mathcal{M}_{akk'},
\end{equation*}
then for any $r>\left(\frac{1}{2-2^{\theta}}+3\right)\rho$ and $$\sup_{n\in \mathbb{N}^*}|I_{n}(0)|e^{2rn^{\theta}} < 1 ,$$ one has
\begin{equation}\label{029}
\sup_{\left|\left|z\right|\right|_{r,\theta}<1}\left|\left|X_H\right|\right|_{r,\theta}\leq C(r,\rho,\theta)||H||_{\rho},
\end{equation}
where $C(r,\rho,\theta)$ is a positive constant depending on $r,\rho$ and $\theta$ only.
\end{lem}
\begin{proof}
Letting
\begin{equation*}
\widetilde{B}_{akk'}=\left(\prod_{n\in\mathbb{N}^*}n^{\frac12\left(2a_n+k_n+k'_n\right)}\right)B_{akk'}
\end{equation*}
and noting that
\begin{equation*}
\left|\widetilde{B}_{akk'}\right|\geq \left|B_{akk}\right|,
\end{equation*}
then following the proof of (5.21) in Lemma 5.2 in \cite{CLSY}, we finish the proof of (\ref{029}).
\end{proof}

\subsection{Iteration and Convergence}

Now we give the precise
set-up of iteration parameters. Let $s\geq1$ be the $s$-th KAM
step.
 \begin{itemize}
 \item[] $\rho_0=\rho,$\ $r\geq \frac{100\rho}{2-2^{\theta}}$,
 \item[]$\delta_{s}=\frac{\rho}{s^2}$,

 \item[]$\rho_{s+1}=\rho_{s}+3\delta_s$,

 \item[]$\epsilon_s=\epsilon_{0}^{(\frac{3}{2})^s}$, which dominates the size of
 the perturbation,

 \item[]$\lambda_s=e^{-C(\theta)(\ln{\frac{1}{\epsilon_{s+1}}})^{\frac{4}{\theta+4}}}$,

 \item[]$ \eta_{0}=1.1-\sup_{n\in\mathbb{N}^*}\omega_n,\eta_{s+1}=\frac{1}{20}\lambda_s\eta_s,$

 \item[]$d_0=0,\,d_{s+1}=d_s+\frac{1}{\pi^2(s+1)^2}$,

 \item[]$D_s=\left\{(z_n)_{n\in\mathbb{N}^*}:\frac{1}{2}+d_s\leq|z_n|e^{rn^{\theta}}\leq1-d_s\right\}$.
 \end{itemize}
Denote the complex cube of size $\lambda>0$:
\begin{equation*}\label{088}
\mathcal{C}_{\lambda}({V^*})=\left\{\left(V_n\right)_{n\in\mathbb{N}^*}\in\mathbb{C}^{\mathbb{N}^*}:|V_n-V^*_n|\leq \lambda\right\}.
\end{equation*}

\begin{lem}{\label{E2}}
Suppose $H_{s}=N_{s}+R_{s}$ is real analytic on $D_{s}\times\mathcal{C}_{\eta_{s}}(V^*_{s})$,
where $$N_{s}=\sum_{n\in\mathbb{N}^*}\lambda_{n,s}(V)|z_n|^2$$ is a normal form with
 \begin{equation}\label{020101}
 \lambda_{n,s}(V)=\sqrt{n^2+\widetilde V_{n,s}(V)}
 \end{equation}satisfying
\begin{eqnarray}
\label{089}&&\sqrt{n^2+\widetilde V_{n,s}(V^*_s)}=n+\omega_n,\\
\label{090}&&\left|\left|\frac{\partial \widetilde{V}_s}{{\partial V}}-I\right|\right|_{l^{\infty}\rightarrow l^{\infty}}<d_s\epsilon_{0}^{\frac{1}{10}},
\end{eqnarray}
and $R_{s}=R_{0,s}+R_{1,s}+R_{2,s}$ satisfying
\begin{eqnarray}
\label{091}&&||R_{0,s}||_{\rho_{s}}^{+}\leq \epsilon_{s},\\
\label{092}&&||R_{1,s}||_{\rho_{s}}^{+}\leq \epsilon_{s}^{0.6},\\
\label{093}&&||R_{2,s}||_{\rho_{s}}^{+}\leq (1+d_s)\epsilon_0.
\end{eqnarray}
Assume that $\omega=(\omega_n)_{n\in\mathbb{N}^*}$ satisfies the nonresonant conditions (\ref{005}) and (\ref{005.}). Then for all $V\in\mathcal{C}_{\eta_{s}}(V_{s}^*)$ satisfying $\widetilde V_{s}(V)\in\mathcal{C}_{\lambda_s}(\omega)$, there exist a real analytic symplectic coordinate transformation
$\Phi_{s+1}:D_{s+1}\rightarrow D_{s}$ satisfying
\begin{eqnarray}
\label{094}&&\left|\left|\Phi_{s+1}-id\right|\right|_{(r,\theta)}\leq \epsilon_{s}^{0.5},\\
\label{095}&&\left|\left|D\Phi_{s+1}-I\right|\right|_{(r,\theta)\rightarrow(r,\theta)}\leq \epsilon_{s}^{0.5},
\end{eqnarray}
such that for
$H_{s+1}=H_{s}\circ\Phi_{s+1}=N_{s+1}+R_{s+1}$, the same assumptions as above are satisfied with `$s+1$' in place of `$s$', where $\mathcal{C}_{\eta_{s+1}}(V_{s+1}^*)\subset\widetilde V_{s}^{-1}(\mathcal{C}_{\lambda_s}(\omega))$ and
\begin{equation}\label{096}
\left|\left|\widetilde{V}_{s+1}-\widetilde{V}_{s}\right|\right|_{\infty}\leq\epsilon_{s}^{0.5},
\end{equation}
\begin{equation}\label{097}
\left|\left|V_{s+1}^*-V_{s}^*\right|\right|_{\infty}\leq2\epsilon_{s}^{0.5},
\end{equation}
where
\begin{equation*}
\left|\left|\Phi_{s+1}-id\right|\right|_{(r,\theta)}:=\sup_{z\in D_{s+1}}\left|\left|\left(\Phi_{s+1}-id\right)(z)
\right|\right|_{r,\theta}.
\end{equation*}
\end{lem}

\begin{proof}
In the step $s\rightarrow s+1$, there is saving of a factor
\begin{equation}\label{19010501}
e^{-\delta_{s}\left(\sum_{n\in\mathbb{N}^*}(2a_n+k_n+k'_n)n^{\theta}-2n_1^{\theta}\right)}.
\end{equation}
By (\ref{001.}), one has
\begin{equation*}
(\ref{19010501})\leq e^{-\left(2-2^\theta\right)\delta_{s}\left(\sum_{i\geq3}n_i^{\theta}\right)}.
\end{equation*}
Recalling after this step, we need
\begin{eqnarray*}
&&||R_{0,s+1}||_{\rho_{s+1}}^{+}\leq \epsilon_{s+1},\\
&&||R_{1,s+1}||_{\rho_{s+1}}^{+}\leq \epsilon_{s+1}^{0.6}.
\end{eqnarray*}
Consequently, in $R_{i,s}\ (i=0,1)$, it suffices to eliminate the nonresonant monomials $\mathcal{M}_{akk'}$ for which
\begin{equation*}
e^{-\left(2-2^\theta\right)\delta_{s}\left(\sum_{i\geq3}n_i^{\theta}\right)}\geq\epsilon_{s+1},
\end{equation*}
that is
\begin{equation}\label{098}
\sum_{i\geq3}n_i^{\theta}\leq\frac{s^2}{(2-2^\theta)\rho}\ln\frac{1}{\epsilon_{s+1}}:=B_s.
\end{equation}
On the other hand, by Remark \ref{013001.} one has
\begin{equation*}\label{013101}
(n_3^*)^{\theta}\leq n_3^{\theta}\leq \sum_{i\geq 3}n_i^{\theta}.
\end{equation*}
Hence, we assume that
\begin{equation*}
n_3^*\leq B_s^{\frac1{\theta}}:=\mathcal{N}_s.
\end{equation*}
We finished the truncation step.

Now we get lower bound on the right hand side of (\ref{063.}) and (\ref{063}) respectively.
Let \begin{equation*}M_s\sim \left(\frac{B_s}{\ln B_s}\right)^{\frac{2}{\theta+2}},\end{equation*} then we have
\begin{eqnarray}
&&\nonumber\prod_{n\in\mathbb{N}^*\atop n\neq n_1^*,n_2^*}
\left(\frac{1}{1+(k_n-k'_n)^2n^6}\right)^4\\
\nonumber&=&e^{-4\sum_{n\leq M_s,n\neq n_1^*,n_2^*}\ln\left({1+(k_n-k'_n)^2n^6}\right)-4\sum_{M_s<n\leq \mathcal{N}_s,n\neq n_1^*,n_2^* }\ln\left({1+(k_n-k'_n)^2n^6}\right)}\\
\nonumber&\geq& e^{-C(\theta)\left(M_s\ln B_s+(M_s^{-{\theta}}\ln M_s)B_s\right)}\\
\nonumber&\geq& e^{-C(\theta)\left(M_s\ln B_s+M_s^{-\frac{\theta}{2}}B_s\right)}\\
\nonumber&\geq& e^{-C(\theta)B_s^{\frac{3}{\theta+3}}}\\
\nonumber&\geq& e^{-C(\rho,\theta)s^{\frac{6}{\theta+3}} \left(\ln{\frac{1}{\epsilon_{s+1}}}\right)^{\frac{3}{\theta+3}}}\\
\label{103}&>&e^{-C(\rho,\theta)\left(\ln{\frac{1}{\epsilon_{s+1}}}\right)^{\frac{4}{\theta+4}}}=\sigma_s,
\end{eqnarray}
where the last inequality is based on $\epsilon_0$ is small enough, $C(\theta)$ is a positive constant depending on $\theta$ and $C(\rho,\theta)$ is a positive constant depending on $\rho$ and $\theta$.

In view of (\ref{013002}) and (\ref{098}), one has
\begin{equation*}
\sum_{i\geq 1}(n_i^*)^{\theta}\leq (11^{\theta}+2)B_s.
\end{equation*}
Let $$\widetilde{\mathcal{N}}_s=(11^{\theta}+2)^{\frac1{\theta}}\mathcal{N}_s,$$
then following the proof of (\ref{103}), we have
\begin{eqnarray}\label{103.}
\prod_{n\in\mathbb{N}^*}
\frac{1}{1+(k_n-k'_n)^2n^5}>\sigma_s.
\end{eqnarray}

Assuming $\widetilde{\omega}\in \mathcal{C}_{\sigma_s}(\omega)$ and from the lower bound (\ref{103}) and (\ref{103.}), the relation (\ref{005}) and (\ref{005.}) remain true if we substitute $\widetilde{\omega}$ for $\omega$. Moreover, there is analyticity on $\mathcal{C}_{\sigma_s}(\omega)$. The transformations $\Phi_{s+1}$ is obtained as the time-1 map $X_{F_s}^{t}|_{t=1}$ of the Hamiltonian
vector field $X_{F_s}$ with $F_s=F_{0,s}+F_{1,s}$. Taking $\rho=\rho_s$, $\delta=\delta_s$ in Lemma \ref{N3}, we get
\begin{eqnarray}\label{104}
\left|\left|F_{i,s}\right|\right|_{\rho_s+\delta_s}^{+}\leq \frac{1}{\gamma^3}\cdot e^{C(\theta)\delta_s^{-\frac5\theta}}\left|\left|R_{i,s}\right|\right|_{\rho_s}^{+},
\end{eqnarray}
where $i=0,1$. By Lemma \ref{N2}, we get
\begin{equation}\label{105}
||F_{i,s}||_{\rho_s+2\delta_s}\leq\frac{C(\theta)}{\delta_s^2}||F_{i,s}||
_{\rho_s+\delta_s}^{+}.
\end{equation}
Combining (\ref{091}), (\ref{092}), (\ref{104}) and (\ref{105}), we get
\begin{equation*}\label{106}
||F_{s}||_{\rho_s+2\delta_s}\leq\frac{C(\theta)}{\gamma^3\delta_s^2}e^{C(\theta)\delta_s^{-\frac5\theta}}(\epsilon_{s}+\epsilon_{s}^{0.6}).
\end{equation*}
By Lemma \ref{H6}, we get
\begin{eqnarray}
\sup_{||z||_{r,\theta}<1}\left|\left|X_{F_s}\right|\right|_{r,\theta}\nonumber
&\leq&C(r,\rho,\theta)||F_{s}||_{\rho_s+2\delta_s}\nonumber\\
&\leq&\frac{C(r,\rho,\theta)}{\gamma^3\delta_s^2}e^{C(\theta)\delta_s^{-\frac5\theta}}(\epsilon_{s}+\epsilon_{s}^{0.6})
\nonumber\\
\label{107}\nonumber&\leq&\epsilon_{s}^{0.55},
\end{eqnarray}
where noting that $0<\epsilon_0\ll1$ small enough and depending on $r,\rho,\theta$ only.

Since $\epsilon_{s}^{0.55}\ll\frac{1}{\pi^2(s+1)^2}=d_{s+1}-d_s$, we have $\Phi_{s+1}:D_{s+1}\rightarrow D_{s}$ with
\begin{equation*}\label{108}
\left|\left|\Phi_{s+1}-id\right|\right|_{(r,\theta)}\leq\sup_{z\in D_{s+1}}\left|\left|X_{F_s}\right|\right|_{r,\theta}\leq\epsilon_{s}^{0.55}<\epsilon_{s}^{0.5},
\end{equation*}
which is the estimate (\ref{094}). Moreover, by Cauchy estimate we get
\begin{equation*}\label{109}
\left|\left|DX_{F_s}-I\right|\right|_{(r,\theta)\rightarrow(r,\theta)}\leq\frac{1}{d_s}\epsilon_{s}^{0.55}<\epsilon_{s}^{0.5},
\end{equation*}
and thus the estimate (\ref{095}) follows.

Moreover, under the assumptions (\ref{091})-(\ref{093}) at stage $s$, we get from (\ref{0861}), (\ref{0862}) and (\ref{0863}) that
\begin{eqnarray*}
||R_{0,s+1}||_{\rho_{s+1}}^{+}
&\leq& e^{\frac{s^{\frac{20}{\theta}}}{\rho^{\frac{10}{\theta}}}}
\left(\epsilon_{0}^{\left(\frac{3}{2}\right)^s}+\epsilon_{0}^{0.9\left(\frac{3}{2}\right)^{s-1}}\right)\left(\epsilon_{0}^{\left(\frac{3}{2}\right)^s}+\epsilon_{0}^{1.8\left(\frac{3}{2}\right)^{s-1}}\right)
\leq\epsilon_{s+1},\\
||R_{1,s+1}||_{\rho_{s+1}}^{+}
&\leq& e^{\frac{s^{\frac{20}{\theta}}}{\rho^{\frac{20}{\theta}}}}\left( \epsilon_{0}^{\left(\frac{3}{2}\right)^s}+\epsilon_{0}^{1.8\left(\frac{3}{2}\right)^{s-1}}\right)
\leq\epsilon_{s+1}^{0.6},\\
\end{eqnarray*}
and
\begin{eqnarray*}
||R_{2,s+1}||_{\rho_{s+1}}^{+}
&\leq& ||R_{2,s}||_{\rho_{s}}^{+}+e^{\frac{s^{\frac{20}{\theta}}}{\rho^{\frac{20}{\theta}}}}
\left(\epsilon_{0}^{\left(\frac{3}{2}\right)^s}+\epsilon_{0}^{0.6\left(\frac{3}{2}\right)^{s}}\right)\\
&\leq&(1+d_s)\epsilon_0+2e^{\frac{s^{\frac{20}{\theta}}}{\rho^{\frac{20}{\theta}}}}
\epsilon_{0}^{0.6\left(\frac{3}{2}\right)^s}\\
&\leq&(1+d_{s+1})\epsilon_0,
\end{eqnarray*}
which are just the assumptions (\ref{091})-(\ref{093}) at stage $s+1$.

Define
\begin{equation*}
\Lambda_s(V)=(\Lambda_{n,s}(V))_{n\in\mathbb{N}^*}
\end{equation*}with
\begin{equation*}
\Lambda_{n,s}(V)=\lambda_{n,s}(V)-n.
\end{equation*}For any $n\in\mathbb{N}^*$, now we would like to prove
\begin{equation}\label{020102}
\Lambda_{s}\left(\mathcal{C}_{\frac1{10}\sigma_s\eta_s}(V_s^*)\right)\subseteq \mathcal{C}_{\sigma_s}(\omega).
\end{equation}
In view of (\ref{020101}), one has
\begin{equation*}
\Lambda_{n,s}(V)=\frac{\widetilde{V}_{n,s}(V)}{\sqrt{n^2+\widetilde{V}_{n,s}(V)}+n}.
\end{equation*}
Hence
\begin{equation*}
\left|\Lambda_{n,s}(V)\right|\leq \frac1{n}\left|\widetilde{V}_{n,s}(V)\right|,
\end{equation*}
where noting that
\begin{equation*}
\sqrt{n^2+\widetilde{V}_{n,s}(V)}+n\geq n.
\end{equation*}
If $V\in \mathcal{C}_{\frac{\eta_s}{2}}(V_s^*)\subseteq\mathcal{C}_{{\eta_s}}(V_s^*)$ and using Cauchy's estimate, one has
\begin{eqnarray}
\sum_{m\in\mathbb{N}^{*}}\left|\frac{\partial \widetilde{V}_{n,s}}{\partial V_m}(V)\right|
\label{110}&\leq& \frac{2}{\eta_s}\sup_{\mathcal{C}_{\frac{\eta_s}{2}}(V_s^*)}\left|\widetilde{V}_{n,s}\right|<\frac{10}{ \eta_s}.
\end{eqnarray}
Let $V\in \mathcal{C}_{\frac{1}{10}\sigma_s\eta_s}(V_s^*)\subseteq\mathcal{C}_{\frac{\eta_s}{2}}(V_s^*)$, then
\begin{eqnarray*}
&&\left|\Lambda_{n,s}(V)-\omega_n\right|\\
&=&\left|\sqrt{n^2+\widetilde{V}_{n,s}(V)}-\sqrt{n^2+\widetilde{V}_{n,s}(V_s^*)}\right|\\
&=&\frac{\left|\widetilde{V}_{n,s}(V)-\widetilde{V}_{n,s}(V_s^*)\right|}{\sqrt{n^2+\widetilde{V}_{n,s}(V)}+\sqrt{n^2+\widetilde{V}_{n,s}(V_s^*)}}\\
& \leq&\frac{2}{3n}\cdot\sup_{\mathcal{C}_{\frac{1}{10}\sigma_s\eta_s}(V_s^*)}\left|\left|\frac{\partial \widetilde{V}_s}{\partial V}\right|\right|_{l^{\infty}\rightarrow l^{\infty}}||V-V_s^*||_{\infty}\\
&<&\frac2{3n}\cdot10 \eta_s^{-1}\cdot\frac{1}{10}\sigma_s\eta_s\qquad \ \mbox{(in view of (\ref{110}))}\\
&=&\frac{2\sigma_s}{3n},
\end{eqnarray*}
which finishes the proof of (\ref{020102}).

Note that
\begin{eqnarray}
\nonumber\left|\left(\prod_{m\in\mathbb{N}^*}m^{2a_m}\right)nB^{(n)}_{a00}\right|
\nonumber&\leq& ||R_{1,s+1}||_{\rho_{s+1}}^+e^{2\rho_{s+1}\left(\sum_{m\in\mathbb{N}^*}a_mm^{\theta}+n^{\theta}-m_1^{\theta}\right)}\\
&<&\nonumber\epsilon_{0}^{0.6\left(\frac{3}{2}\right)^{s}}e^{2\rho_{s+1}\left(\sum_{m\in\mathbb{N}^*}a_mm^{\theta}+n^{\theta}-m_1^{\theta}\right)},
\end{eqnarray}
which implies
\begin{equation*}\label{111}
\left|B^{(n)}_{a00}\right|<\frac1{n}\epsilon_{0}^{0.6\left(\frac{3}{2}\right)^{s}}e^{2\rho_{s+1}\left(\sum_{m\in\mathbb{N}^*}a_mm^{\theta}+n^{\theta}-m_1^{\theta}\right)}
\end{equation*}
Assuming further
\begin{equation}\label{112}
I_{m}(0)\leq e^{-2rm^{\theta}}
\end{equation}
and for any $s$,
\begin{equation}\label{113}
\rho_s<\frac{1}{2}r,
\end{equation}
we obtain
\begin{eqnarray}
\nonumber\left|\sum_{a\in\mathbb{N}^{\mathbb{N}^*}}B^{(n)}_{a00}\mathcal{M}_{a00}\right|
\nonumber&\leq & \frac1{n} \epsilon_{0}^{0.6\left(\frac{3}{2}\right)^{s}}\sum_{a\in\mathbb{N}^{\mathbb{N}^*}}e^{2\rho_{s+1}\left(\sum_{m\in\mathbb{N}^*}a_mm^{\theta}+n^{\theta}-m_1^{\theta}\right)}\prod_{m\in\mathbb{N}^*}I_{m}(0)^{a_m}\\
\nonumber&\leq& \frac1{n} \epsilon_{0}^{0.6\left(\frac{3}{2}\right)^{s}}\sum_{a\in\mathbb{N}^{\mathbb{N}^*}}e^{2(\rho_{s+1}\left(\sum_{m\in\mathbb{N}^*}a_mm^{\theta}\right)}\prod_{m\in\mathbb{N}^*}I_{m}(0)^{a_m}\\
\nonumber&\leq& \frac1{n} \epsilon_{0}^{0.6\left(\frac{3}{2}\right)^{s}}\sum_{a\in\mathbb{N}^{\mathbb{N}^*}}e^{\sum_{m\in\mathbb{N}^*}2\left(\rho_{s+1}-r\right)a_mm^{\theta}}\qquad \mbox{(in view of (\ref{112}))}\\
\nonumber&\leq& \frac1{n} \epsilon_{0}^{0.6\left(\frac{3}{2}\right)^{s}}\sum_{a\in\mathbb{N}^{\mathbb{N}^*}}e^{-r\left(\sum_{m\in\mathbb{N}^*}a_mm^{\theta}\right)}\qquad \mbox{(in view of (\ref{113}))}\\
\nonumber&\leq& \frac1{n}\epsilon_{0}^{0.6\left(\frac{3}{2}\right)^{s}}\prod_{m\in\mathbb{N}^*}\left(1-e^{-r m^{\theta}}\right)^{-1} \\
\label{114}&\leq&\frac1{n}\epsilon_{0}^{0.6\left(\frac{3}{2}\right)^{s}}\left(\frac{1}{r}\right)^{C(\theta){r^{-\frac{1}{\theta}}}},
\end{eqnarray}
i.e.
\begin{equation*}
\left|\lambda_{n,s+1}-\lambda_{n,s}\right|<\frac1{n}\epsilon_{0}^{0.6\left(\frac{3}{2}\right)^{s}}\left(\frac{1}{r}\right)^{C(\theta){r^{-\frac{1}{\theta}}}}.
\end{equation*}
Noting that
\begin{equation*}
\lambda_{n,s+1}-\lambda_{n,s}=\sqrt{n^2+\widetilde {V}_{n,s+1}}-\sqrt{n^2+\widetilde {V}_{n,s}}=\frac{\widetilde{V}_{n,s+1}-\widetilde{V}_{n,s}}{\sqrt{n^2+\widetilde {V}_{n,s+1}}+\sqrt{n^2+\widetilde {V}_{n,s}}},
\end{equation*}
then one has
\begin{eqnarray}
\nonumber\left|\widetilde{V}_{n,s+1}-\widetilde{V}_{n,s}\right|
\nonumber&<&\left(\frac{1}{r}\right)^{C(\theta){r^{-\frac{1}{\theta}}}}\epsilon_{0}^{0.6\left(\frac{3}{2}\right)^{s}}<\epsilon_{s}^{0.5},
\end{eqnarray}
which verifies (\ref{096}). Further applying Cauchy's estimate on $\mathcal{C}_{\sigma_s\eta_s}(V_s^*)$, one gets
\begin{eqnarray}
\nonumber\sum_{m\in\mathbb{N}^*}\left|\frac{\partial \widetilde{V}_{n,s+1}- \widetilde{V}_{n,s}}{\partial V_m}\right|
\nonumber&\leq& \frac{10\epsilon_{s}^{0.5}}{\sigma_s\eta_s}\\
\nonumber&\leq& \left(\frac{10}{\eta_s}\right) e^{C(\rho,\theta)(\ln\frac{1}{\epsilon_{s+1}})^{\frac{4}{4+\theta}}-0.5\ln\frac{1}{\epsilon_{s+1}}}\\
\nonumber&\leq&\left(\frac{1}{\eta_s}\right) e^{-\frac14\ln\frac{1}{\epsilon_{s+1}}}\\
\label{116}&=& \frac{1}{\eta_s}\epsilon_{0}^{\frac{1}{4}\left(\frac{3}{2}\right)^{s+1}}.
\end{eqnarray}
Since
\begin{equation*}
\eta_{s+1}=\frac{1}{20}\sigma_s\eta_s,
\end{equation*}
it follows that
\begin{eqnarray}
\nonumber\eta_{s+1}&=& \eta_s e^{-C(\rho,\theta)\left(\ln\frac{1}{\epsilon_0}\right)^{\frac{4}{4+\theta}}\left(\frac32\right)^{\frac{4(s+1)}{4+\theta}}}\\
\nonumber&\geq& \eta_se^{-C(\rho,\theta)\left(\ln\frac{1}{\epsilon_0}\right)\left(\frac32\right)^{\frac{5s}{5+\theta}}}\ \ \ \mbox{(for $\epsilon_0$ small enough)}\\
\label{117}&=&\eta_s\epsilon_0^{C(\rho,\theta)\left(\frac32\right)^{\frac{5s}{5+\theta}}},
\end{eqnarray}
and hence by iterating (\ref{117}) implies
\begin{eqnarray}
\nonumber\eta_{s}&\geq&\eta_0\epsilon_{0}^{C(\rho,\theta)\sum_{i=0}^{s-1}\left(\frac{3}{2}\right)^{\frac{5i}{5+\theta}}}\\
\nonumber&=&\eta_0\epsilon_{0}^{C(\rho,\theta)\frac{\left(\frac32\right)^{\frac{5s}{5+\theta}}-1}{\left(\frac32\right)^{\frac{5}{5+\theta}}-1}}\\
\nonumber&>&\eta_0\epsilon_{0}^{C(\rho,\theta)\left(\frac{3}{2}\right)^{\frac{5s}{5+\theta}}}\\
\label{118}&\geq&\epsilon_{0}^{\frac{1}{100}\left(\frac{3}{2}\right)^{s}}\ \ \ \mbox{(for $\epsilon_0$ small enough)}.
\end{eqnarray}
On $ \mathcal{C}_{\frac{1}{10}\sigma_s\eta_s}(V_s^*)$ and for any $n$, we deduce from (\ref{116}), (\ref{118}) and the assumption (\ref{090}) that
\begin{eqnarray*}
\sum_{m\in\mathbb{N}^*}\left|\frac{\partial \widetilde{V}_{n,s+1}}{\partial V_m}-\delta_{nm}\right|
&\leq&\sum_{m\in\mathbb{N}^*}\left|\frac{\partial\widetilde{V}_{n,s+1}}{\partial V_m}-\frac{\partial\widetilde{V}_{n,s}}{\partial V_m}\right|+\sum_{m\in\mathbb{N}^*}\left|\frac{\partial \widetilde{V}_{n,s}}{\partial V_m}-\delta_{nm}\right|\\
&\leq&\epsilon_{0}^{\left(\frac{3}{8}-\frac{1}{100}\right)\left(\frac{3}{2}\right)^{s}}+d_s\epsilon_{0}^{\frac{1}{10}}\\
&<&d_{s+1}\epsilon_{0}^{\frac{1}{10}},
\end{eqnarray*}
and consequently
\begin{equation*}\label{119}
\left|\left|\frac{\partial \widetilde{V}_{s+1}}{{\partial V}}-I\right|\right|_{l^{\infty}\rightarrow l^{\infty}}<d_{s+1}\epsilon_{0}^{\frac{1}{10}},
\end{equation*}
which verifies (\ref{090}) for $s+1$.

Finally, we will freeze $\omega$ by invoking an inverse function theorem. Consider the following functional equation
\begin{equation*}\label{120}
\widetilde{V}_{n,s+1}(V_{s+1}^*)=n\omega_n+\omega_n^2,
\end{equation*}
and
\begin{equation*}
V_{s+1}^*\in \mathcal{C}_{\frac{1}{10}\sigma_s\eta_s}(V_s^*).
\end{equation*}
From (\ref{090}) and the standard inverse function theorem implies (\ref{120}) having a solution $V_{s+1}^*$, which verifies (\ref{089}) for $s+1$. Noting that
\begin{equation*}\label{121}
V_{s+1}^*-V_s^*=(I-\widetilde{V}_{s+1})(V^*_{s+1})-(I-\widetilde{V}_{s+1})({V^*_s})+(\widetilde{V}_s-\widetilde{V}_{s+1})(V_s^*),
\end{equation*}
and  using (\ref{096}), (\ref{090}), one has
\begin{equation*}\label{122}
||V_{s+1}^*-V_s^*||_{\infty}\leq (1+d_{s+1})\epsilon_{0}^{\frac{1}{10}}||V_{s+1}^*-V_s^*||_{\infty}+\epsilon_s^{0.5}<2\epsilon_s^{0.5}\ll \sigma_s\eta_s,
\end{equation*}
which verifies (\ref{097}) and completes the proof of the iterative lemma.
\end{proof}

We are now in a position to prove Theorem \ref{thm}.
 \begin{proof}To apply iterative lemma with $s=0$, set
\begin{equation*}
V_{n,0}=n\omega_n+\omega_n^2,\hspace{12pt}\widetilde{V}_0=id,\hspace{12pt}\epsilon_0=\epsilon,
\end{equation*}
and consequently (\ref{089})--(\ref{093}) with $s=0$ are satisfied. Hence, the iterative lemma applies, and we obtain a decreasing
sequence of domains $D_{s}\times\mathcal{C}_{\eta_{s}}(V_{s}^*)$ and a sequence of
transformations
\begin{equation*}
\Phi^s=\Phi_1\circ\cdots\circ\Phi_s:\hspace{6pt}D_{s}\times\mathcal{C}_{\eta_{s}}(V_{s}^*)\rightarrow D_{0}\times\mathcal{C}_{\eta_{0}}(V_{0}^*),
\end{equation*}
such that $H\circ\Phi^s=N_s+R_s$ for $s\geq1$. Moreover, the
estimates (\ref{094})--(\ref{097}) hold. Thus we can show $V_s^*$ converge to a limit $V_*$ with the estimate
\begin{equation*}
||V_*-\omega||_{\infty}\leq\sum_{s=0}^{\infty}2\epsilon_{s}^{0.5}<\epsilon^{0.4},
\end{equation*}
and $\Phi^s$ converge uniformly on $D_*\times\{V_*\}$, where $D_*=\{(z_n)_{n\in\mathbb{N}^*}:\frac{2}{3}\leq|z_n|e^{rn^{\theta}}\leq\frac{5}{6}\}$, to $\Phi:D_*\times\{V_*\}\rightarrow D_0$ with the estimates
\begin{eqnarray}
\nonumber&&||\Phi-id||_{\left(r,\theta\right)}\leq \epsilon^{0.4},\\
\nonumber&&||D\Phi-I||_{(r,\theta)\rightarrow(r,\theta)}\leq \epsilon^{0.4}.
\end{eqnarray}
Hence
\begin{equation*}\label{123}
H_*=H\circ\Phi=N_*+R_{2,*},
\end{equation*}
where
\begin{equation*}\label{124}
N_*=\sum_{n\in\mathbb{N}^*}(n+\omega_n)|z_n|^2
\end{equation*}
and
\begin{equation*}\label{125}
||R_{2,*}||_{10\rho}^{+}\leq\epsilon^{0.4}.
\end{equation*}
\end{proof}
\begin{rem}\label{030201}
By (\ref{029}), the Hamiltonian vector field $X_{R_{2,*}}$ is a bounded map from ${G}^{r,\theta}$ into ${G}^{r,\theta}$. Taking
\begin{equation*}\label{126}
I_n(0)=\frac{3}{4}e^{-2rn^{\theta}},
\end{equation*}
we get an invariant torus $\mathcal{T}$ with frequency $(n+\omega_n)_{n\in\mathbb{N}^*}$ for ${X}_{H_*}$. Moreover, we deduce the torus $\Phi(\mathcal{T})$ is linearly stable from the fact that (\ref{123}) is a normal form of order 2 around the invariant torus.
\end{rem}
\section{Application to the nonlinear wave equation}

We study equation (\ref{maineq0}) as an infinite dimensional hamiltonian system. As the
phase space one may take, for example, the product of the usual Sobolev spaces
$\mathcal{P}= H^1_0 ([0, \pi])\times L^2([0, \pi])$ with coordinates $u$ and $v = u_t$. Then the hamiltonian of (\ref{maineq0}) is
\begin{equation*}H = \frac12
\langle v, v\rangle + \frac12
\langle Au, u\rangle +\frac{\epsilon}4
\int_0^{\pi}
u^4 dx,
\end{equation*}
where $A = -d^2/dx^2+V(x)*$ and $\langle \cdot,\cdot\rangle $ denotes the usual scalar
product in $L^2$ . The hamiltonian equations of motions are
\begin{equation*}u_t = \frac{\partial H}{\partial v}=v,\quad
 v_t = -\frac{\partial H}{\partial u}=-Au-u^3,
\end{equation*}
hence they are equal to (\ref{maineq0}).

To rewrite it as a hamiltonian in infinitely many coordinates we make the ansatz
\begin{equation*}
u=\mathcal{S}q=\sum_{n\in\mathbb{N}^*}\frac{q_n}{\sqrt{\lambda_n}}\phi_n,\quad
v=\mathcal{S'}p=\sum_{n\in\mathbb{N}^*}\sqrt{\lambda_n}{p_n}\phi_n
\end{equation*}
where $$\phi_n=\sqrt{\frac{2}{\pi}}\sin nx$$ for $n=1,2,\dots$ are the normalized Dirichlet eigenfunctions
of the operator $A$ with eigenvalues
\begin{equation*}\lambda_n=\sqrt{n^2+V_n}.
\end{equation*}

We obtain the Hamiltonian
\begin{equation}\label{0001'}
H=\Lambda+G=\frac{1}2\sum_{n\in\mathbb{N}^*}\lambda_n(p_n^2+q_n^2)+
\frac{\epsilon}4\sum_{i,j,k,l\in\mathbb{N}^*\atop \pm i\pm j\pm k\pm l=0}G_{ijkl}q_iq_jq_kq_l,
\end{equation}
with
\begin{equation}\label{022601}
G_{ijkl}=\frac{1}{\sqrt{\lambda_i\lambda_j\lambda_k\lambda_l}}\int_0^{\pi}\phi_i\phi_j\phi_k\phi_ldx.
\end{equation}
We introduce the complex coordinates
\begin{equation*}
z_n=\frac{1}{\sqrt{2}}\left(q_n+\textbf{i}p_n\right),\qquad \bar z_n=\frac{1}{\sqrt{2}}\left(q_n-\textbf{i}p_n\right),
\end{equation*}with $\textbf{i}=\sqrt{-1}$. Then the Hamiltonian (\ref{0001'}) is turned into
\begin{equation}\label{012501}
H(z,\bar z)=\sum_{n\in\mathbb{N}^*}\lambda_nz_n\bar z_n+\frac{\epsilon}{16}\sum_{i,j,k,l\in\mathbb{N}^*\atop\pm i\pm j\pm k\pm l=0}G_{ijkl}(z_i+\bar z_i)(z_j+\bar z_j)(z_k+\bar z_k)(z_l+\bar z_l)
\end{equation}
 Then the  Hamiltonian (\ref{012501}) has the form of
\begin{equation*}
H(z,\bar z)=N(z,\bar z)+R(z,\bar z),
\end{equation*}
where
\begin{equation*}
N(z,\bar z)=\sum_{n\in\mathbb{N}^*}\lambda_n|z_n|^2,
\end{equation*}
and
 \begin{equation*}
 R(z,\bar z)=\frac{\epsilon}{16}\sum_{i,j,k,l\in\mathbb{N}^*\atop\pm i\pm j\pm k\pm l=0}G_{ijkl}(z_i+\bar z_i)(z_j+\bar z_j)(z_k+\bar z_k)(z_l+\bar z_l).
 \end{equation*}
In view of (\ref{022601}), one has
\begin{equation*}
\left|\left|R\right|\right|_{\rho}\leq C\epsilon.
\end{equation*}
Applying Theorem \ref{thm} and Remark \ref{030201}, we finish the proof of Theorem \ref{Thm}.
\section{Measure Estimate and technical lemma}
\begin{lem}\label{050601}
Let the set
$$\Pi=\left[0,1\right]\times\left[0,1/2\right]\times\cdots[0,1/n]\times\cdots$$ with probability measure. Then there exists a subset $\Pi_{\gamma}\subset\Pi$ with
\begin{equation}
\mbox{meas}\ \Pi_{\gamma}\leq C\gamma,
\end{equation}
where $C$ is a positive constant, such that for any $\omega\in\Pi\setminus\Pi_{\gamma}$, the inequalities (\ref{005}) and (\ref{005.}) holds.
\end{lem}
\begin{proof}
Define the resonant set $\mathcal{R}_l$ by
\begin{eqnarray}
 \mathcal{R}_l=\left\{\omega:\left\| \sum_{n\in\mathbb{N}^*}l_n\omega_{n}\right\|< \gamma \prod_{n\in \mathbb{N}^*}\frac{1}{1+l_n^2n^{5}}\right\},
\end{eqnarray}
and
\begin{equation}\label{020701}
\mathcal{R}_1=\bigcup_{l\in\mathbb{Z}^{\mathbb{N}^*}}\mathcal{R}_l.
\end{equation}
Then following the proof of Lemma 4.1 in \cite{BJFA2005}, one has
\begin{equation}\label{012806.}
\mbox{meas}\ \mathcal{R}_1\leq C_1\gamma,
\end{equation}
where $C_1$ is a positive constant.

Define the resonant set $\widetilde{\mathcal{R}}_l$ (where considering $l=k-k'$) by
\begin{eqnarray}
 \widetilde{\mathcal{R}}_l=\left\{\omega:\left\| \sum_{n\in\mathbb{N}^*}l_n\omega_{n}\right\|< \frac{\gamma^3}{16} \prod_{n\in \mathbb{N}^*\atop n\neq n_1^*,n_2^*}\left(\frac{1}{1+l_n^2n^{6}}\right)^4\right\},
\end{eqnarray}
Then one has
\begin{equation}\label{012802}
\mbox{meas}\ \widetilde{\mathcal{R}}_l\leq \frac{m\gamma^3}{16}\prod_{n\geq m\atop n\neq n_1^*,n_2^*}\left(\frac{1}{1+l_n^2n^{6}}\right)^4,
\end{equation}
where $l_j=0$ with $1\leq j\leq m-1$ and $l_m\neq 0$.

Note that
\begin{equation}
 \sum_{n\in\mathbb{N}^*}l_n\omega_{n}= \sum_{n\in\mathbb{N}^*,\atop n\neq n_1^*,n_2^*}l_n\omega_n+\sigma_{n_1^*}\omega_{n_1^*}+\sigma_{n_2^*}\omega_{n_2^*},
\end{equation}
where $\sigma_{n_1^*},\sigma_{n_2^*}\in\left\{-1,1\right\}$. Hence, if $\omega\in\Pi\setminus\mathcal{R}_1$ (where $\mathcal{R}_1$ is defined in (\ref{020701})) and
\begin{equation}\label{012801}
n_2^*\geq \frac4{\gamma}\prod_{n\in\mathbb{N}^*\atop n\neq n_1^*,n_2^*}\left(1+l_n^2n^5\right),
\end{equation}
then
\begin{eqnarray*}
&&\left|\left|\sum_{n\in\mathbb{N}^*,\atop n\neq n_1^*,n_2^*}l_n\omega_n+\sigma_{n_1^*}\omega_{n_1^*}+\sigma_{n_2^*}\omega_{n_2^*}\right|\right|\\
&\geq&\left|\left|\sum_{n\in\mathbb{N}^*,\atop n\neq n_1^*,n_2^*}l_n\omega_n\right|\right|-\left|\left|\omega_{n_1^*}+\omega_{n_2}^*\right|\right|\\
&\geq&  \gamma \prod_{n\in \mathbb{N}^*\atop n\neq n_1^*,n_2^*}\frac{1}{1+l_n^2n^{5}}-\frac2{n_2^*}\\
&\geq &\frac{\gamma}{2} \prod_{n\in \mathbb{N}^*\atop n\neq n_1^*,n_2^*}\frac{1}{1+l_n^2n^{5}},
\end{eqnarray*}
where the last inequality is based on (\ref{012801}). Hence, we always assume
\begin{equation}\label{012803}
n_2^*< \frac4{\gamma}\prod_{n\in\mathbb{N}^*\atop n\neq n_1^*,n_2^*}\left(1+l_n^2n^5\right):=A(l).
\end{equation}
If
\begin{equation}
n_1^*=n_2^*,
\end{equation}
then one has
\begin{equation}
n_1^*<A(l).
\end{equation}
If $n_1^*>n_2^*$, then noting that
\begin{equation}
n_1^*\leq \sum_{i\geq 2}n_i^*,
\end{equation}
and
\begin{equation}
\sum_{i\geq 2}n_i^*=\sum_{n\in\mathbb{N}^*\atop n\neq n_1^*}\left|l_n\right|n\leq \prod_{n\in\mathbb{N}^*\atop n\neq n_1^*}(1+l^2_nn^5)
\end{equation}
which implies
\begin{equation}\label{012804}
n_1^*< A(l)\left(\prod_{n\in\mathbb{N}^*\atop n\neq n_1^*,n_2^*}(1+l^2_nn^5)\right)=\frac{4}{\gamma}\left(\prod_{n\in\mathbb{N}^*\atop n\neq n_1^*,n_2^*}(1+l^2_nn^5)\right)^2:=B(l).
\end{equation}
Then define the resonant set
\begin{equation}
\mathcal{R}_2=\bigcup_{l\in\mathbb{Z}^{\mathbb{N}^*}\atop
n_1^*<B(l),n_2^*<{A(l)}}\widetilde {\mathcal{R}}_l
\end{equation}
In view of (\ref{012802}), (\ref{012803}), (\ref{012804}) and following the proof of (\ref{012806.}), one has
\begin{equation}\label{012807}
\mbox{meas}\ \mathcal{R}_2\leq C_3\gamma,
\end{equation}
where $C_3$ is a positive constant.

Let
$$\Pi_{\gamma}=\Pi\setminus(\mathcal{R}_1\bigcup\mathcal{R}_2),$$
then one has
\begin{equation}\label{012806}
\mbox{meas}\ \Pi_{\gamma}\leq C\gamma,
\end{equation}
and for any $\omega\in\Pi\setminus \Pi_{\gamma}$, the inequalities (\ref{005}) and (\ref{005.}) holds.

\end{proof}

\begin{lem}\label{031902}
The following estimate holds
\begin{equation*}
\left(\prod_{n\in\mathbb{N}^*\atop n\neq n_1^*,n_2^*}\left({1+(k_n-k'_n)^2n^6}\right)^4\right)\times e^{-\delta\left(\sum_{n\in\mathbb{N}^*}(2a_n+k_n+k'_n)n^{\theta}-2n_1^{\theta}\right)}\leq e^{C(\theta)\delta^{-\frac5\theta}},
\end{equation*}
where $C(\theta)$ is a positive constant depending on $\theta$ only.
\end{lem}
\begin{proof}By a direct calculation, one has
\begin{eqnarray*}&&\left(\prod_{n\in\mathbb{N}^*\atop n\neq n_1^*,n_2^*}\left({1+(k_n-k'_n)^2n^6}\right)^4\right)\times e^{-\delta\left(\sum_{n\in\mathbb{N}^*}(2a_n+k_n+k'_n)n^{\theta}-2n_1^{\theta}\right)}\\
\nonumber &\leq&e^{4\sum_{n\in\mathbb{N}^*, n\neq n_1^*,n_2^*}\ln(1+(k_n-k'_n)^2n^6)}\times e^{-{\delta}{\left(2-2^{\theta}\right)}\left(\sum_{i\geq3}n_i^{\theta}\right)}\qquad\nonumber\mbox{(in view of (\ref{001.}))}\\
&\leq&\nonumber e^{1+48\left(\sum_{{n\in\mathbb{N}^{*}:k_n\neq k'_n\atop n\neq n_1^*,n_2^*}}\ln\left(|k_n-k'_n|n\right)\right)-{\delta}{\left(2-2^{\theta}\right)}\left(\sum_{i\geq3}\left(n_i^*\right)^{\theta}\right)}\qquad\nonumber(\mbox{in view of Remark \ref{013001.}})\\
&=&\nonumber e^{1+\sum_{{n\in\mathbb{N}^{*}:k_n\neq k'_n\atop n\neq n_1^*,n_2^*}}\left(48\ln\left(|k_n-k'_n|n\right)-{\delta}{\left(2-2^{\theta}\right)}\left|k_n-k_n'\right|n^{\theta}\right)}\\
&\leq&\nonumber e^{1+\sum_{|n|\leq N:k_n\neq k'_n}\left(48\ln\left(|k_n-k'_n|n\right)-\tilde{\delta}|k_n-k_n'|^{\theta}n^{\theta}\right)}\\
&&+\nonumber\frac{16e}{\gamma^3} ||{R_0}||_{\rho}^+e^{\sum_{n>N:k_n\neq k'_n}\left(48\ln\left(|k_n-k'_n|n\right)-\tilde{\delta}|k_n-k_n'|^{\theta}n^{\theta}\right)}\\
&&\nonumber \mbox{$\left( \mbox{where}\ \tilde{\delta}={\delta}{\left(2-2^{\theta}\right)} \ \mbox{and}\ N=\left(\frac{48}{\theta \tilde{\delta} }\right)^{4/\theta}\ \mbox{and noting}\ \left|k_n-k_n'\right|^{\theta}\leq \left|k_n-k_n'\right|\right)$}\\
&\leq&\nonumber e^{\left(1+\frac{48}{\theta \tilde{\delta}}\right)^{4/\theta}\cdot \frac{48}{\theta}\ln\left(\frac{48}{\theta \tilde{\delta}}\right)}+e\qquad{\mbox{(in view of  (\ref{065}) and (\ref{066}) below)}}\\
&\leq &e^{C(\theta)\delta^{-\frac5\theta}},
\end{eqnarray*}
where $C(\theta)$ is a positive constant depending on $\theta$ only.

For $0<\delta\ll1$, it is easy to verify the following two facts that:

 (1) let $f(x)=48\ln x-\tilde\delta x^{\theta}$, and then
\begin{equation}\label{065}
\max_{x\geq 1} f(x)=f\left(\left(\frac{48}{\theta\tilde\delta}\right)^{1/\theta}\right)= 48\ln\left(\left(\frac{48}{\theta\tilde\delta}\right)^{1/\theta}\right)-\frac {48}{\theta}\leq \frac{48}{\theta}\ln\left(\frac{48}{\theta\tilde\delta}\right);
\end{equation}

(2) for $k_n\neq k_n'$ and $n>N=\left(\frac{48}{\theta\tilde\delta}\right)^{4/\theta}$, one has
\begin{equation}\label{066}
48\ln\left(|k_n-k'_n|n\right)-\tilde\delta\left(|k_n-k_n'|^{\theta}n^{\theta}\right)<0.
\end{equation}
\end{proof}

\end{document}